\pdfoutput=1                      

\documentclass{amsart}

\usepackage[accsupp]{axessibility}    

\usepackage{amsmath}
\usepackage{amssymb}
\usepackage{amsthm}
\usepackage{color}
\usepackage{verbatim}
\usepackage{tabularx}

\usepackage{mathrsfs}
\newtheorem{theorem}{Theorem}

\newtheorem{conjecture}[theorem]{Conjecture}
\newtheorem{corollary}[theorem]{Corollary}

\newtheorem{lemma}[theorem]{Lemma}

\newtheorem{proposition}[theorem]{Proposition}
\newtheorem{remark}[theorem]{Remark}

\newcommand{\R}{\mathbb{R}}

\newcommand{\metric}{\langle \, , \, \rangle}
\newcommand{\disp}{\displaystyle}
\newcommand{\ra}{\rightarrow}

\newcommand{\eps}{\varepsilon}

\newcommand{\Sph}{\mathbb{S}}
\newcommand{\di}{\mathrm{d}}

\newcommand{\Ricc}{\mathrm{Ric}}
\newcommand{\Riem}{\mathrm{Riem}}

\newcommand{\vol}{\mathrm{vol}}
\newcommand{\lip}{\mathrm{Lip}}
\newcommand{\loc}{\mathrm{loc}}

\newcommand{\Ric}{\mathrm{Ric}}
\newcommand{\RR}{\mathbb{R}}

\newcommand{\Sp}{\mathrm{Sp}}

\DeclareMathOperator{\diver}{div\,}

\DeclareMathOperator{\Hess}{\mathrm{Hess}}

\DeclareMathOperator{\tr}{\mathrm{Tr}}

\renewcommand{\div}{\mathrm{div}}

\makeatletter
\newcommand*\owedge{\mathpalette\@owedge\relax}
\newcommand*\@owedge[1]{
	\mathbin{
		\ooalign{
			$#1\m@th\bigcirc$\cr
			\hidewidth$#1\m@th\wedge$\hidewidth\cr
		}
	}
}
\makeatother

\begin{document}

\title[$\varphi$-CPE metrics]{Einstein-type structures, Besse's conjecture and a uniqueness result for a $\varphi$-CPE metric in its conformal class}
\author{Giulio Colombo \and Luciano Mari \and Marco Rigoli}

\begin{abstract}
In this paper, we study an extension of the CPE conjecture to manifolds $M$ which support a structure relating curvature to the geometry of a smooth map $\varphi : M \to N$. The resulting system, denoted by \eqref{phCPE}, is natural from the variational viewpoint and describes stationary points for the integrated $\varphi$-scalar curvature functional restricted to metrics with unit volume and constant $\varphi$-scalar curvature. We prove both a rigidity statement for solutions to \eqref{phCPE} in a conformal class, and a gap theorem characterizing the round sphere among manifolds supporting \eqref{phCPE} with $\varphi$ a harmonic map.   
\end{abstract}

\maketitle

\section{Introduction}

The Critical Point Equation, from now on the \ref{CPE} equation, is the Euler-Lagrange equation of the Hilbert-Einstein action on the space of Riemannian metrics with unit volume and constant scalar curvature on a compact manifold. It has been introduced, in the attempt to more efficiently identify Einstein metrics, by A. Besse in his treatise, \cite{besse87}, to which we refer for details. From now on $(M,\metric)$ will denote a connected  Riemannian manifold of dimension $m\geq2$. The \ref{CPE} equation is the following system of PDEs:
\begin{equation} \label{CPE}\tag{CPE}
	\Hess(w) - w \left( \Ricc - \frac{S}{m-1} \, \metric \right) = T
\end{equation}
for some $w\in C^\infty(M)$ (we shall not be interested in further constraints on $w$, see \cite{besse87}). Here, $\Ricc$, $T$ and $S$ denote, respectively, the Ricci, the traceless Ricci tensors and the scalar curvature of $(M,\metric)$.

Besse's conjecture (or at least a version of it) can be stated as follows:
\begin{conjecture} \label{besse_conj}
	If $(M,\metric)$ is compact, $S$ is constant and $w \not \equiv -1$ is a smooth solution to \eqref{CPE} on $M$, then $(M,\metric)$ is Einstein.
\end{conjecture}

Constant solutions are easily handled: if $w$ is a constant different from $-1$, then \eqref{CPE} implies that $M$ is Einstein (indeed, Ricci flat if $w \neq 0$); on the other hand, if $w \equiv -1$ then \eqref{CPE} is equivalent to $S \equiv 0$. 

In order to derive \eqref{CPE} we have assumed from the very beginning that $S$ is constant, but it is worth to observe that the mere existence of a solution $w$ of \eqref{CPE} implies the constancy of $S$ (this will be shown in Proposition \ref{prop_DS} below, in a more general setting). Taking this into account, with the aid of a result of Obata \cite{obata62} we may state the following form of Besse's conjecture, which up to removing the case of constant $w$ is equivalent to the original formulation:
\begin{conjecture}
	If $(M,\metric)$ is compact and $w$ is a non-constant solution to \eqref{CPE} on $M$, then $(M,\metric)$ is isometric to a standard sphere.
\end{conjecture}

Indeed, tracing \eqref{CPE} we get
\begin{equation} \label{trCPE1}
	\Delta w + \frac{S}{m-1} w = 0,
\end{equation}
thus integrating by parts yields
	\begin{equation}\label{eq_intide}
	\int_M \frac{S}{m-1} w^2 = \int_M |\nabla w|^2 \, .
	\end{equation}
Since $w$ is non-constant and $S$ is constant it follows that $S>0$. If $(M,\metric)$ is Einstein, \eqref{CPE} reduces to 
\[
	\Hess(w) = - \frac{S}{m(m-1)} w \metric \, ,
\]
whence using Theorem A of \cite{obata62} we obtain that $M$ is isometric to a round sphere.

There are a number of partial results on Besse's  conjecture, and we list some of them. Precisely, the conjecture is true if one of the following sets of assumptions is satisfied on $M$ compact:
\begin{itemize}
	\item[i)] $\frac{S}{m-1} \not\in \Sp(-\Delta)$ (cf. \cite[Proposition 4.47]{besse87});
	\item[ii)] $(M,\metric)$ is locally conformally flat and the solution $w$ to \eqref{CPE} is not unique (Lafontaine \cite{laf83}). The result was improved by removing the second assumption (LaFontaine-Rozoy \cite{blr} for $m=3$, and Chang-Hwang-Yun \cite{chy12});
	\item[iii)] $w\geq -1$ (Hwang \cite{hwang00}). It is worth to observe that this result follows from the very interesting identity 
	$$
		\div(T(\nabla w,\,\cdot\,)^\sharp) = (1+w)|T|^2
	$$
	and from the fact that, if $w \not \equiv -1$, $\{x\in M : w(x) = -1\}$ has measure $0$. Here $^\sharp$ is the musical isomorphism (in Lemma \ref{lem_divT} below we shall generalize the above identity);
	\item[iv)] $\div\Riem = 0$ (Yun-Chang-Hwang  \cite{ych14});
	\item[v)] $(M,\metric)$ is Bach flat (Qing-Yuan \cite{qy13});
	\item[vi)] $m=4$ and $\diver W^+\equiv0$, where $W^+$ is the self-dual part of the Weyl tensor $W$ (Barros-Leandro-Ribeiro \cite{blr15});
	\item[vii)] $(M,\metric)$ is conformally Einstein (Barros-Evangelista \cite{be18});
	\item[viii)] $m \ge 5$ and the radial Weyl curvature $i_{\nabla w}W = 0$ (Baltazar-Barros-Batista-Viana \cite{bbbv20});
	\item[ix)] $m\geq 3$ and condition
	\begin{equation}\label{cond_balta}
		|W| \leq \sqrt{\frac{m}{2(m-2)}} \left[ \frac{S}{\sqrt{m(m-1)}} - 2|T| \right]
	\end{equation}
	holds (Baltazar \cite{bal20});
	\item[x)] $m=3$ and $\Ricc\geq0$ (He \cite{he21}).
\end{itemize}

Furthermore, in a very recent preprint, Hwang and Yu \cite{hwangyu} showed that the CPE conjecture holds if $\metric$ has positive isotropic curvature.

The \ref{CPE} equation is strictly related to the vacuum static equation
\begin{equation} \label{VSE}\tag{VSE}
	\Hess(w) - w \left( \Ricc - \frac{S}{m-1} \metric \right) = 0 \, ,
\end{equation}
where we consider smooth solutions $w\not\equiv0$. Indeed, observe that if $M$ admits two different solutions $w_0,w_1$ to \eqref{CPE}, then for any $t \in \R$ the function $w_t = (1-t)w_0 + t w_1$ solves \eqref{CPE} and its $t$-derivative $w_1-w_0$ is a non-trivial solution to \eqref{VSE}. Although \eqref{VSE} can be seen as the Euler-Lagrange equation of an action functional over a certain space of metrics with constant scalar curvature,  similarly to \eqref{CPE} one verifies that the sole existence of a non-trivial solution $w$ to \eqref{VSE} on $M$ implies that the scalar curvature is constant, see \cite{cem13}.

Suppose $m\geq 3$. A recent result of Herzlich, \cite{herz16}, provides a nice class of solutions to \eqref{VSE}. Indeed, he shows that if $X$ is a conformal vector field on an Einstein manifold $(M,\metric)$, then
$$
	w = \div X
$$
is a solution to \eqref{VSE} which is non-trivial so long as $X$ is not a Killing field. 
%
%
In fact, when $M$ is compact with $\partial M\neq\emptyset$, assuming the existence of $X$ as above and constancy of $S$, Miao and Tam, \cite{mt17}, were able to prove, under some further assumptions, that if $\div X$ solves \eqref{VSE} then $(M,g)$ is Einstein, providing a partial converse of Herzlich result.


An interesting problem related to the \ref{VSE} equation is that of the local scalar curvature rigidity; that is, to look for domains $\Omega$ in $(M^m,\metric)$ such that for each metric $g$ inducing the same metric as $\metric$ on $\partial\Omega$ and such that $S_g \geq S_{\metric}$ on $\Omega$ and $H_g = H_{\metric}$ on $\partial\Omega$, $H$ the mean curvature of $\partial\Omega$ with respect to the inward pointing normal, there exists $\eps>0$ for which the condition
$$
	\| g - \metric \|_{C^2(\Omega)} < \eps
$$
implies the existence of a diffeomorphism $\psi : \overline{\Omega} \to \overline{\Omega}$ with the property that $\metric = \psi^\ast g$ and $\psi \equiv \mathrm{id}$ on $\partial\Omega$. Here and in what follows, agreeing with most of the literature we adopt the convention that the mean curvature $H_h$ of $\partial \Omega$ in a given metric $h$ is normalized and taken with respect to the the inward pointing direction, namely, 
	\begin{equation}\label{eq_conveH}
	H_h = \frac{\diver_h\nu_h}{m-1}, \qquad \text{$\nu_h$  the outward unit normal to $\partial \Omega \hookrightarrow (\overline{\Omega},h)$}.
	\end{equation}
The above problem is closely related to the well known conjecture of Min-Oo on $\Sph^m_+$, which in its full generality was disproved by Brendle, Marques and Neves, \cite{bmn11}. However, Hang and Wang, \cite{hw06}, obtained a positive answer to a weaker form of Min-Oo's conjecture, proving the scalar curvature rigidity among conformal metrics for the round hemisphere $\Sph^m_+$. The result has been recently extended by Qing and Yuan \cite{qy16}, Yuan \cite{y17} and Barbosa, Mirandola and Vitorio \cite{bmv} to a manifold with a solution of \eqref{VSE}, that they more simply call a vacuum static space. In particular, we have the following elegant result, that we rephrase to facilitate its comparison with our Theorem \ref{thm_phS_uniq} below:
\begin{theorem}[Corollary 7 in \cite{bmv}] \label{thm_qy}
	Let $(M,g)$ be a complete vacuum static space with $w\not\equiv0$ solution to \eqref{VSE} and scalar curvature $S_g \ge 0$. Let 
	\[
	\Omega \subset \big\{ x : w(x) > 0 \big\}
	\] 
be a relatively compact, connected open set with smooth boundary. If $\tilde g$ is conformal to $g$ on $\overline{\Omega}$ and it satisfies
	$$
		\begin{cases}
			S_{\tilde g} \geq S_g & \text{on } \, \Omega \\
			\tilde g \equiv g & \text{on } \, \partial\Omega \\
		\end{cases}
	$$
then 
	\[
	\tilde g \ge g \quad \text{on } \, \Omega, \qquad H_{\tilde g} \le H_g \quad \text{on } \, \partial \Omega.
	\]
Furthermore, both inequalities are strict unless $\tilde g=g$ on $\overline\Omega$.
\end{theorem}
Taking traces in \eqref{VSE}, notice that $w$ is a positive solution of $\mathscr{L}w = 0$ on the set
	\[
	\Omega_+ = \big\{ x \in M \ : \ w(x) > 0\big\},
	\]
where 
	\[
	\mathscr{L} = - \Delta - \frac{S_g}{m-1},
	\]
therefore $\lambda_1(\mathscr{L},\Omega) \ge \lambda_1(\mathscr{L},\Omega_+) \ge 0$, where $\lambda_1$ is the bottom of the spectrum with Dirichlet boundary conditions. We also stress that, in \cite{qy16}, the authors prove that $\Omega_+$ is ``maximal'' for the validity of the result. \\
\par
%
%
%
Recent years saw a rising interest in manifolds whose curvatures relate to properties of a smooth map $\varphi : (M,\metric) \to (N,\metric_N)$ into a target Riemannian space. One of the first instances of such interplay is the work of Buzano \cite{mul12}, where the author investigated the Ricci flow coupled with the harmonic map flow. Solitons for the flow are characterized by the system
\begin{equation}
	\begin{cases}
		\Ricc^\varphi + \Hess(f) = \lambda \metric \\
		\tau(\varphi) = \di\varphi(\nabla f),
	\end{cases}
\end{equation}
where $\alpha, \lambda \in \R$, $\tau(\varphi)$ is the tension field of $\varphi$ (see \cite{el88} and the beginning of Section \ref{sec_thm_Sph}) and $\Ric^\varphi$ is the $\varphi$-Ricci tensor
\begin{equation}
	\Ricc^\varphi = \Ricc - \alpha\varphi^\ast\metric_N \, .
\end{equation}
See also Wang \cite{wang16} for related results. For constant $f$, the above reduces to the harmonic-Einstein system 
\begin{equation} \label{harm_Ein}
	\begin{cases}
		\Ricc^\varphi = \lambda\metric \\
		\tau(\varphi) = 0 \, 
	\end{cases}
\end{equation}
which extends the notion of Einstein manifolds to possibly nonconstant $\varphi$ (as in the Einstein case, by \cite[Proposition 2.15]{acr21}, if $m \ge 3$ then $\lambda$ is necessarily constant). The interest in \eqref{harm_Ein} is made even more evident if we rewrite the first identity as
	\begin{equation}\label{eq_Geinstein}
	G + \Lambda \metric = \alpha \overline{T},
	\end{equation}
where $G$ is the Einstein tensor of $M$, 
	\[
	\Lambda = \frac{m-2}{2} \lambda
	\]
and $\overline{T}$ is the stress-energy tensor\footnote{Notice that in \cite{acr21}, after equation (1.5), there is a typo in the definition of the stress-energy tensor.} of the map $\varphi$:
	\[ 
	\overline{T} \doteq \varphi^\ast\metric_N - \frac{|\di\varphi|^2}{2} \metric \, .
	\]
Notice that we did not use the fact that $\metric$ is Riemannian. Hence, in a Lorentzian setting, solutions to \eqref{harm_Ein} with $\alpha >0$ correspond to solutions to the Einstein field equation with cosmological constant $\Lambda$ and source the wave map $\varphi$, up to a normalization constant. The fact that the left hand side of \eqref{eq_Geinstein} is divergence free forces $\overline{T}$ to be divergence free as well, which is equivalent to the vanishing of the $1$-form $\langle \tau(\varphi),\di \varphi \rangle_N$. The harmonicity of $\varphi$ is then a sufficient condition for the compatibility of the system. Other examples and more detailed  discussions can be found in \cite{a21,acr21}.	\\
\par 
%
%
%
%
%
In what follows, we shall investigate the CPE problem in the more general setting just mentioned. To properly define the system corresponding to \eqref{CPE}, first recall the obvious definitions of the $\varphi$-scalar curvature and traceless $\varphi$-Ricci tensor:
\begin{equation}
	S^\varphi \doteq \tr \Ricc^\varphi = S - \alpha |\di \varphi|^2, \qquad T^\varphi \doteq \Ricc^\varphi - \frac{S^\varphi}{m} \metric \, .
\end{equation}
Associated to $\varphi$ and $\alpha$, further ``curvature'' tensors that we shall call $\varphi$-curvatures will be introduced below at due time. For more information we refer to \cite{a21,acr21,mr20} where we justify the various concepts and prove a number of results.

We formally introduce the $\varphi$-CPE equation, that we shall also call a $\varphi$-CPE structure, by requiring the existence of $w\in C^\infty(M)$ solving
\begin{equation} \label{phCPE}\tag{$\varphi$-CPE}
	\begin{cases}
		\Hess(w) - w\left( \Ricc^\varphi - \frac{S^\varphi}{m-1} \metric \right) = T^\varphi \\
		(1+w) \tau(\varphi) = -\di\varphi(\nabla w)
	\end{cases}
\end{equation}
The system \eqref{phCPE} will be justified from a variational viewpoint in Section \ref{sec_variational}. We remark that, as in the case of constant $\varphi$, the mere validity of \eqref{phCPE} implies that $S^\varphi$ is constant, see Proposition \ref{prop_DS} below. If $w> -1$, performing the change of variable
\begin{equation} \label{fw_chvar}
	f = -\log(1+w)
\end{equation}
\eqref{phCPE} becomes equivalent to the Einstein-type structure
\begin{equation} \label{Ein-type}
	\begin{cases}
		\Ricc^\varphi + \Hess(f) - \mu \di f \otimes \di f = \lambda \metric \\
		\tau(\varphi) = \di\varphi(\nabla f)
	\end{cases}
\end{equation}
with the choices $\mu=1$ and
$$
	\lambda(x) = \frac{S^\varphi}{m-1} \left( 1 - \frac{e^f}{m} \right) \, .
$$
Similarly, if $w>0$ the $\varphi$-VSE equation
\begin{equation} \label{phVSE}\tag{$\varphi$-VSE}
	\begin{cases}
		\Hess(w) - w\left( \Ricc^\varphi - \frac{S^\varphi}{m-1} \metric \right) = 0 \\
		w \tau(\varphi) = -\di\varphi(\nabla w)
	\end{cases}
\end{equation}
falls into the class \eqref{Ein-type} with the choices
	\[
	f = -\log w, \qquad \mu=1, \qquad \lambda(x) = - \frac{S^\varphi}{m-1} \frac{e^f}{m}.
	\]
The importance of the general Einstein-type structure \eqref{Ein-type} is evident. For instance it describes, as special cases, Ricci-harmonic solitons, Ricci solitons, generalized quasi-Einstein manifolds for $\mu = \mu(x)$ and $\lambda = \lambda(x)$, and so on. Moreover it appears quite naturally in several problems coming from Physics, see \cite{a21,acr21,mr20}.

As expected, since \eqref{Ein-type} encompasses a wide range of different structures, its validity does not force, in general, the constancy of $S^\varphi$. However, this is the case for some noticeable examples, for instance when $m \ge 3$ and \eqref{Ein-type} reduces to the harmonic-Einstein system \eqref{harm_Ein}. This parallels Schur's Theorem and is a simple but revealing instance pointing out that the theory of harmonic-Einstein manifolds, that is, of those Riemannian manifolds supporting a solution to \eqref{harm_Ein}, has many analogies with that of Einstein manifolds. Another example which is relevant for us is \eqref{phCPE} in dimension $m \ge 3$, see Proposition \ref{prop_DS} below. For results in this direction we refer to \cite{acr21}.\\
\par
Our first theorem relates to Theorem \ref{thm_qy}. To state it we need some further piece of notation. Given a metric $g$ on $M$, we set $[g]$ to denote its conformal class. If $\varphi : (M,g) \to (N,\metric_N)$ and $\tilde g \in [g]$, we denote with a tilde quantities referred to $\tilde g$, and we let
$$
	\tilde\varphi \ : \ (M,\tilde g) \to (N,\metric_N) \, , \qquad \tilde\varphi(x) \doteq \varphi(x).
$$
Also, having fixed an origin $o\in M$ we let $B_r = \{ x \in M : \mathrm{dist}_g(x,o) < r\}$.

\begin{theorem} \label{thm_phS_uniq}
	Let $(M,g)$ be a complete manifold of dimension $m\geq 3$ possessing an Einstein-type structure as in \eqref{Ein-type} with $\mu \in \R^+$ and $\lambda = \lambda(x) \in C^\infty(M)$. Assume $S^\varphi \ge 0$, and let $\Omega$ be a smooth, connected open set satisfying
	\begin{equation}
		\Omega \subset \{ x \in M : f(x) < 0 \}.
	\end{equation}
Let $\tilde g\in[g]$ satisfy
	\begin{equation} \label{parOm_cond}
		i) \; \tilde S^{\tilde\varphi} \geq S^\varphi \ \ \text{ on } \, \Omega\, ; \qquad ii) \; \tilde g \equiv g  \ \ \text{ on } \, \partial \Omega.
	\end{equation}
Assume that
	\begin{equation} \label{lambda_eps}
		\lambda(x) \geq \frac{S^\varphi}{\mu m(m-1)} \left[1 + \mu(m-1) - e^{\mu f}\right] + \eps \qquad \text{on } \, \Omega,
	\end{equation}
	for some $\eps\geq 0$. If either 
	\begin{equation} \label{phS_vol_1}
		\eps > 0 \qquad \text{and} \qquad \liminf_{r\to \infty} \frac{\log |B_r\cap\Omega|}{r} = 0
	\end{equation}
	or
	\begin{equation} \label{phS_vol_0}
		\eps = 0 \qquad \text{and} \qquad \liminf_{r\to \infty} \frac{|B_r\cap\Omega|}{r^2} = 0 \, ,
	\end{equation}
then
	\begin{equation}\label{eq_ine0}
	\tilde g \ge g \quad \text{on } \, \Omega, \qquad H_{\tilde g} \le H_g \quad \text{on } \, \partial \Omega, 
	\end{equation}
where $H_{\tilde g}$ and $H_g$ are the mean curvatures of $\partial \Omega$ in the inward direction in the metrics $\tilde g$ and $g$, respectively. Furthermore, inequalities in \eqref{eq_ine0} are strict unless $\tilde g \equiv g$ on $\overline{\Omega}$.
	\end{theorem}
Note that, for $\mu = 1$ and
$$
f = -\log(1+w), \qquad	\lambda(x) = \frac{S^\varphi}{m-1}\left( 1 - \frac{e^f}{m} \right)
$$
we are exactly in the case of the \ref{phCPE} equation and \eqref{lambda_eps} is satisfied with $\eps=0$. Thus,  we have
\begin{corollary} \label{cor_phS_uniq}
	Let $(M,g)$ be a complete manifold of dimension $m\geq 3$ possessing a $\varphi$-CPE structure \eqref{phCPE} with $S^\varphi \ge 0$. Let
	$$
		\Omega \subset \{ x \in M : w(x) > 0 \}
	$$
be a connected open set with smooth boundary, For $\tilde g \in [g]$ assume the validity of \eqref{parOm_cond} and suppose that, for a fixed origin $o\in M$,
	$$
		\liminf_{r\to \infty} \frac{|B_r \cap \Omega|}{r^2} = 0 \, .
	$$
Then,
	\begin{equation}\label{eq_ine}
	\tilde g \ge g \quad \text{on } \, \Omega, \qquad H_{\tilde g} \le H_g \quad \text{on } \, \partial \Omega, 
	\end{equation}
and the inequalities are strict unless $\tilde g \equiv g$ on $\overline{\Omega}$.
\end{corollary}

A result corresponding to Corollary \ref{cor_phS_uniq} can be formulated for \eqref{phVSE}. We leave the statement and details to the interested reader. 

\begin{remark}
\emph{Strictly speaking Corollary \ref{cor_phS_uniq} is a consequence of Theorem \ref{thm_phS_uniq} in case $w>-1$ due to the change of variable \eqref{fw_chvar}. However, the same argument of the proof of Theorem \ref{thm_phS_uniq} applies directly \emph{mutatis mutandis} to Corollary \ref{cor_phS_uniq} without requiring that $w> -1$.
}
\end{remark}

The idea of the proof of Theorem \ref{thm_phS_uniq} is different from the one used by Barbosa, Mirandola and Vitorio in \cite{bmv}, and for relatively compact sets $\Omega$ it gives an alternative approach to their Theorem \ref{thm_qy} which is closer to the arguments in \cite{qy16}. We stress that  the open set $\Omega_f$ is not required to be relatively compact, and indeed, avoiding relative compactness of $\Omega_f$ is a bit of a subtle technical point in the proof. Our aim is achieved by using a special case of an analytic result of independent interest, see Lemma \ref{L2_liou} below.\par

The second theorem we are going to prove extends Baltazar's recent result mentioned above, \cite{bal20}, to the system \eqref{phCPE}. At the same time, we streamline part of the proof, highlighting the role played by a Kazdan-Warner type identity in Lemma \ref{lem_divT} below. We introduce the $\varphi$-Weyl tensor
\begin{equation}
	W^\varphi = \Riem - \frac{1}{m-2} A^\varphi \owedge \metric \, , \quad m\geq 3
\end{equation}
where $\Riem$ is the Riemann tensor, $\owedge$ is the ``parrot'' (Kulkarni-Nomizu) product and $A^\varphi$ is the $\varphi$-Schouten tensor
\begin{equation}\label{def_Schou}
	A^\varphi = \Ricc^\varphi - \frac{S^\varphi}{2(m-1)} \metric \, .
\end{equation}
We can express $W^\varphi$ in terms of the usual Weyl tensor $W$ via the equation
\begin{equation}
	W^\varphi = W + \frac{\alpha}{m-2} F \owedge \metric \, , \quad m\geq 3
\end{equation}
with
\begin{equation}
	F = \varphi^\ast\metric_N - \frac{|\di\varphi|^2}{2(m-1)} \metric \, .
\end{equation}
%
Note that the $\varphi$-Weyl tensor has the same symmetries of $\Riem$, in particular, it satisfies the first Bianchi identity. However, in general it is not totally trace free. 

We are ready to state our second main result. Notice that, for constant $\varphi$, \eqref{Sph_pinch} below becomes condition \eqref{cond_balta}. 

\begin{theorem} \label{thm_Sph}
	Let $(M,\metric)$ be a compact manifold of dimension $m\geq 3$ with a $\varphi$-CPE structure as in \eqref{phCPE} for some non-constant function $w$ and some $\alpha \in \R$. Assume that $\tau(\varphi) = 0$ and that 
	\begin{equation} \label{Sph_pinch}
		\frac{S^\varphi}{2(m-1)} - \frac{\alpha}{m-2}|\di\varphi|^2 \geq \sqrt{\frac{m-2}{2(m-1)}}|W^\varphi| + \sqrt{\frac{m}{m-1}}|T^\varphi|
	\end{equation}
	on $M$. Then $(M,\metric)$ is isometric to the standard sphere $\Sph^m(\kappa) \subseteq \RR^{m+1}$ of constant sectional curvature
	\begin{equation} \label{kappa}
		\kappa = \frac{S^\varphi}{m(m-1)} \, .
	\end{equation}
Moreover, if $\alpha \neq 0$ then $\varphi$ is constant.
	\end{theorem}


\section{Variational derivation of \eqref{phCPE}}\label{sec_variational}

We let $M$ be compact and without boundary, and denote with $\mathscr{M}$ be the set of smooth Riemannian metrics on $M$, endowed with the compact open $C^\infty$ topology. We also fix $(N,\metric_N)$ and denote with $\mathscr{F}$ the set of smooth maps $\varphi : M \to N$, again with the compact open $C^\infty$ topology. We consider the functional $\mathcal{S} : \mathscr{M} \times \mathscr{F} \to \R$ given by
	\[
	\mathcal{S}(g,\varphi) = \int_M S^\varphi_g \di x_g,
	\]
where $\di x_g$, $S^\varphi_g$ are the Riemannian volume and the $\varphi$-scalar curvature of $g$. In this section, we explain how \eqref{phCPE} can be seen as the Euler-Lagrange equation of $\mathcal{S}$ restricted to the subset of metrics and maps $(g,\varphi)$ with unit volume and constant $\varphi$-scalar curvature $S_g^\varphi$. Our treatment parallels the one in \cite[pp. 127-128]{besse87}. To avoid technicalities, we keep the discussion at an informal level and do not describe the function spaces used to justify the properties of the operators which we are going to consider. Given $(h,v) \in T_{(g,\varphi)}(\mathscr{M} \times \mathscr{F}) = T_g\mathscr{M} \times T_\varphi \mathscr{F}$ (notice that $T_\varphi \mathscr{F}$ can be identified with sections of $\varphi^*TN$ via the exponential map), it holds 
	\[
	\Big( \di_{(g,\varphi)}\mathcal{S}\Big)[(h,v)] = \int_M \left[ \dot{S}^\varphi_g(h,v) \di x_g + S^\varphi_g \dot{\di x_g}(h,v) + (S^\varphi_g)'(h,v) \di x_g	\right] \, , 
	\]
where the dot and prime symbols denote, respectively, differentiation with respect to $g$ and $\varphi$ at the point $(g,\varphi)$. Direct computations (cf. Propositions 34 and 35 in \cite{a21}) give
	\[
\begin{array}{lcl}
\dot{S}^\varphi_g(h,v) & = & \disp - \Delta_g \big( \tr_g h\big) + \diver_g\big(\diver_g h\big) - \langle h, \Ric^\varphi_g \rangle_g \\[0.4cm]
\dot{\di x_g}(h,v) & = & \frac{1}{2} \tr_g h \, \di x_g \\[0.4cm]
(S^\varphi)'(h,v) & = & - \alpha (|\di \varphi|^2_g)'(h,v) = -2\alpha \diver_g( \langle \di \varphi, v \rangle_N \big) + 2\alpha \langle \tau(\varphi), v \rangle_N, 
\end{array}	
	\]
whence, integrating by parts, 	
	\begin{equation}\label{eq_ELeq}
	\Big( \di_{(g,\varphi)}\mathcal{S}\Big)[(h,v)] = - \int_M \langle \Ric_g^\varphi - \frac{S_g^\varphi}{2}g, h \rangle_g \di x_g + 2\alpha \int_M \langle \tau_g(\varphi),v \rangle_N \di x_g.
	\end{equation}
Let $\mathscr{M}_1 \subset \mathscr{M}$ be the subset of metrics $g$ with $\vol_g(M) = 1$, and let 
	\[
	\mathscr{G} = \big\{ (g, \varphi) \in \mathscr{M}_1 \times \mathscr{F} \ : \ \text{$S^\varphi_g$ is constant} \big\}. 
	\]
If $(h,v)$ generates a  variation $(g_t,\varphi_t)$ of $(g, \varphi) \in \mathscr{G}$ for which, up to first order, the scalar curvature $S^{\varphi_t}_{g_t}$ is constant on $M$ for each $t$, 
	\[
	\begin{array}{l}
	\left. \frac{\di}{\di t}\right|_{t=0} S^{\varphi_t}_{g_t} = \Big(\di_{(g,\varphi)}S_g^\varphi \Big)[(h,v)] \\[0.4cm]
	\qquad =  - \Delta_g \big( \tr_g h\big) + \diver_g\big(\diver_g h\big) - \langle h, \Ric^\varphi_g \rangle_g -2\alpha \diver_g( \langle \di \varphi, v \rangle_N \big) + 2\alpha \langle \tau(\varphi), v \rangle_N
	\end{array}
	\]
must be constant on $M$, equivalently, 
	\[
	\beta_{(g,\varphi)}[(h,v)] \doteq  \Delta_g \left( \Big( \di_{(g,\varphi)}S_g^\varphi \Big)[(h,v)] \right) = 0.
	\]
Therefore, at least formally, $T_{(g,\varphi)} \mathscr{G}$ can be seen as the set of pairs
	\begin{equation}\label{T_1G}
	(h,v) \in \ker \beta_{(g,\varphi)} \cap \big( T_g \mathscr{M}_1 \times T_\varphi \mathscr{F}\big),
	\end{equation}
where 
	\begin{equation}\label{eq_tgM1}
	T_g \mathscr{M}_1 = \left\{ h'\in S^2(M) : \int_M \tr_g h' \di x_g = 0 \right\}. 
	\end{equation}
We hereafter assume $(g,\varphi) \in \mathscr{G}$.  Computing the adjoint map 
	\[
	\beta_{(g,\varphi)}^* \ : \ C^\infty(M) \to T_g \mathscr{M}_1 \times T_\varphi \mathscr{F},
	\]
we get
	\[
	\begin{array}{l}
	\disp \int_M \langle \beta_{(g,\varphi)}^*(\eta), (h,v)\rangle_g = \disp \int_M \eta \beta_{(g,\varphi)}[(h,v)]\di x_g \\[0.4cm]
	\qquad  \disp = \int_M \Delta_g \eta \Big( \dot{S}^\varphi_g(h,v) + (S^\varphi_g)'(h,v) \Big)\di x_g \\[0.4cm]
	\qquad = \disp \int_M \langle h, - (\Delta_g \Delta_g \eta) g + \Hess_g(\Delta_g \eta) - (\Delta_g \eta) \Ric^\varphi_g \rangle_g \di x_g\\[0.4cm]
	\qquad \quad \disp -2\alpha \int_M \diver_g\big( \langle \di \varphi, v \rangle_N \big) \Delta_g \eta \di x_g + 2 \alpha \int_M \langle \tau_g(\varphi), v\rangle_N \Delta_g \eta \di x_g \\[0.4cm]
	\qquad = \disp \int_M \langle h, - (\Delta_g \Delta_g \eta) g + \Hess_g(\Delta_g \eta) - (\Delta_g \eta) \Ric^\varphi_g \rangle_g \di x_g\\[0.4cm]
	\qquad \quad \disp + 2 \alpha \int_M \langle \di \varphi(\nabla \Delta_g \eta) + (\Delta_g \eta) \tau_g(\varphi), v \rangle_N \di x_g,
	\end{array}
	\]  
and therefore, 
	\[
	\beta_{(g,\varphi)}^*(\eta) = \Big( - (\Delta_g w) g + \Hess_g(w) - w \Ric^\varphi_g, 2\alpha \big( \di \varphi(\nabla w) + w \tau_g(\varphi)\big) \Big) 
	\]
where $w = \Delta_g \eta$. Notice that $\beta_{(g,\varphi)}^*$ is valued in $T_g \mathscr{M}_1 \times T_\varphi \mathscr{F}$, because $(g,\varphi) \in \mathscr{G}$ implies the constancy of $S^\varphi_g$. In view of \eqref{eq_tgM1} and again because of the constancy of $S^\varphi_g$, the Euler-Lagrange equation  \eqref{eq_ELeq} can be written as follows for variations $(h,v) \in T_g \mathscr{M}_1 \times T_\varphi \mathscr{F}$:
	\begin{equation}\label{eq_bellaEL}
	\Big( \di_{(g,\varphi)}\mathcal{S}\Big)[(h,v)] = - \int_M \langle T^\varphi, h \rangle_g \di x_g + 2\alpha \int_M \langle \tau_g(\varphi),v \rangle_N \di x_g.
	\end{equation}
Taking into account that the tangent space $T_g \mathscr{M}_1 \times T_\varphi \mathscr{F}$ decomposes as  
	\[
	T_g \mathscr{M}_1 \times T_\varphi \mathscr{F} = \Big( \ker \beta_{(g,\varphi)} \cap \big( T_g \mathscr{M}_1 \times T_\varphi \mathscr{F} \big) \Big) \oplus^\perp \mathrm{Im} \beta_{(g,\varphi)}^* \, ,
	\]
then from \eqref{eq_bellaEL} we deduce that $(g,\varphi)$ is critical for $\mathcal{S}$ with respect to variations satisfying \eqref{T_1G} if and only if $(T^\varphi, -2\alpha \tau_g(\varphi)) \in \mathrm{Im} \beta_{(g,\varphi)}^*$, that is, if and only if there exists $w \in C^\infty(M)$ that can be written as $w = \Delta_g \eta$ (equivalently, $w$ has mean value zero on $(M,g)$) and satisfies \eqref{phCPE}.

\section{Proof of Theorem \ref{thm_phS_uniq}}

We shall make use of the following Liouville type theorem. Hereafter, 
\[
v_+ \doteq \max\{v,0\} 
\]
is the positive part of a function $v$. 

	\begin{lemma} \label{L2_liou}
		Let $(M,\metric)$ be a complete manifold and $\Omega\subseteq M$ be a connected open set with non-empty boundary. Let $0< w\in C(\Omega)$. For each $\delta>0$, define $w_\delta = w+ \delta$ and let $\{\psi_\delta\}_\delta \subset \lip_\loc(\Omega)$ satisfy
		\begin{equation}\label{ip_technical}
		\psi_\delta \to \psi, \qquad (\psi_\delta)_+ w_\delta \uparrow \psi_+w \qquad \text{locally uniformly in $\Omega$ as $\delta \to 0$,}
		\end{equation}
for some $\psi \in C(\Omega)$. Assume further that $\psi_\delta$ is a weak solution of
		\begin{equation} \label{L2_psi_cond}
			\begin{cases}
			w_\delta^{-2} \diver\big( w_\delta^2 \nabla \psi_\delta \big) \geq (c-f_\delta) \psi_\delta  & \quad \text{on } \, \Omega_\delta \doteq \{ x \in \Omega : \psi_\delta(x) > 0 \} \neq \emptyset \\[0.4cm]
				\displaystyle \limsup_{x\to\partial\Omega}\psi_\delta(x) \leq 0 \, ,
			\end{cases}
		\end{equation}
for some constant $c \ge 0$ and some functions $f_\delta : \Omega \to \R$ satisfying
	\begin{equation}\label{ass_fdelta}
	\begin{array}{l}
	(i) \quad \forall \, K \subset \overline\Omega \ \text{ compact,} \quad \|f_\delta\|_{L^\infty(K)} \le C_K \ \text{ for some constant } \, C_K > 0; \\[0.2cm]
	(ii) \quad f_\delta \to 0 \ \ \text{ pointwise a.e. in } \, \Omega. 
	\end{array}
	\end{equation}
If $\{ x : \psi(x) > 0\} \neq \emptyset$, then the following holds:
	\begin{align}
			\label{L2_vol_1}
			\text{if } \, c > 0 \, , & \text{ then} \qquad \liminf_{r\to \infty} \frac{1}{r} \log \left( \int_{B_r \cap \Omega} w^2 \psi_+^2 \right) > 0 \, ; \\
			\label{L2_vol_0}
			\text{if } \, c = 0 \, , & \text{ then} \qquad \liminf_{r\to \infty} \frac{1}{r^2} \int_{B_r\cap\Omega} w^2 \psi_+^2 > 0  \quad \text{ unless $\psi$ is constant.}
		\end{align}
where $B_r$ is the geodesic ball of radius $r$ in $M$ centered at a fixed origin. 		
	\end{lemma}

\begin{remark}
\emph{The limsup in the second condition of \eqref{L2_psi_cond} is defined as 
\[
\limsup_{x \ra \partial \Omega}\psi_\delta(x) \doteq \inf\Big\{ \sup_{\Omega \backslash \overline{V}} \psi_\delta \ : \ V \ \text{open whose closure in $M$ satisfies } \, \overline{V} \subset \Omega \Big\}.
\]
}
\end{remark}

\begin{remark}
\emph{Notice that the above lemma also applies to relatively compact (connected) domains $\Omega$. If $\Omega$ is relatively compact, the limit relations \eqref{L2_vol_1}, \eqref{L2_vol_0} never hold unless $w \psi_+ \not \in L^2(\Omega)$ (or $\psi$ is constant, if $c=0$).
} 
\end{remark}

\begin{remark}
\emph{The reason why Lemma \ref{L2_liou} is stated for a sequence of approximating solutions $\{\psi_\delta\}$ rather than for a single solution $\psi$ is to allow for limits $\psi$ that may \emph{not} satisfy the boundary condition in \eqref{L2_psi_cond}. This will be crucial in application to the proof of Theorem \ref{thm_phS_uniq}. 
}
\end{remark}

\begin{proof}
We hereafter assume that
	\begin{equation}\label{eq_ass}
	\int_{B_r \cap \Omega} w^2 \psi_+^2 < \infty \qquad \text{for each } \, r > 0,
	\end{equation}
otherwise the desired conclusion is obvious. Note that the second in \eqref{ip_technical} and $\Omega_\delta \neq \emptyset$ imply that $\psi_+ \not \equiv 0$, so we can fix $R_0$ large enough such that $\psi_+ \not \equiv 0$ on $\Omega \cap B_{R_0}$. For every $\delta>0$, choose $\eps_0 = \eps_0(\delta) >0$ small enough so that
	$$
		\Omega_{\delta,\eps} = \{ x \in \Omega : \psi_\delta(x) > \eps \}
	$$
	is non-empty for $\eps \le \eps_0$, and define
	$$
		\psi_{\delta,\eps} = (\psi_\delta -\eps)_+ = \max\{ \psi_\delta - \eps, 0 \} \, .
	$$
By the second condition in \eqref{L2_psi_cond}, $\mathrm{supp} \psi_{\delta,\eps} = \overline{\Omega_{\delta,\eps}}$ does not meet $\partial\Omega$. 
Given $R>R_0$, let $\eta$ be a Lipschitz cut-off function such that
	$$
		\eta \equiv 1 \, \text{ on } \, B_R \, , \qquad \eta \equiv 0 \, \text{ on } \, M \setminus B_{2R} \, , \qquad |\nabla\eta| \leq \frac{1}{R} \, \text{ on } \, M.
	$$
   Let $\alpha\geq 1$ to be suitably chosen later and consider $z_{\delta,\eps} := \eta^{2\alpha}\psi_{\delta,\eps}$. Note that $z_{\delta,\eps}$ is non-negative and  Lipschitz with compact support in $\overline{\Omega_{\delta,\eps}} \subseteq \Omega_\delta$, so it is an admissible test function to be inserted in the weak definition of \eqref{L2_psi_cond} to obtain
	\[
		\begin{array}{lcl}
			\disp \int_\Omega (c-f_\delta) z_{\delta,\eps} \psi_\delta w_\delta^2 & \le & - \disp \int_\Omega \langle \nabla\psi_\delta,\nabla z_{\delta,\eps} \rangle w_\delta^2 \\[0.4cm]
			& = & - \disp \int_\Omega |\nabla\psi_\delta|^2 \mathbf{1}_{\Omega_{\delta,\eps}} \eta^{2\alpha} w_\delta^2 - 2\alpha \int_\Omega \langle\nabla\eta,\nabla\psi_\delta\rangle \eta^{2\alpha-1} \psi_{\delta,\eps} w_\delta^2 \\[0.4cm]
			& \leq & - \dfrac{1}{2} \disp \int_\Omega |\nabla\psi_\delta|^2 \mathbf{1}_{\Omega_{\delta,\eps}} \eta^{2\alpha} w_\delta^2 + 2\alpha^2 \int_\Omega \eta^{2\alpha-2} \psi_{\delta,\eps}^2 w_\delta^2 |\nabla\eta|^2
		\end{array}
	\]
	where in passing from the second to the third line we have used Cauchy-Schwarz and Young's inequalities. Rearranging terms and using H\"older's inequality we arrive at
	\begin{equation}\label{eq_notbad_1}
		\begin{array}{lcl}
			\disp \int_\Omega (c-f_\delta) z_{\delta,\eps} \psi_\delta w_\delta^2 & \leq & \disp \int_\Omega (c-f_\delta) z_{\delta,\eps} \psi_\delta w_\delta^2 + \dfrac{1}{2} \disp \int_\Omega |\nabla\psi_\delta|^2 \mathbf{1}_{\Omega_{\delta,\eps}} \eta^{2\alpha} w_\delta^2 \\[0.4cm]
			& \leq & \disp 2\alpha^2 \int_\Omega \eta^{2\alpha-2} \psi_{\delta,\eps}^2 w_\delta^2 |\nabla\eta|^2 \\[0.4cm]
			& \leq & \disp 2\alpha^2 \left( \int_\Omega \psi_{\delta,\eps}^2 w_\delta^2 \eta^{2\alpha} \right)^{\frac{\alpha-1}{\alpha}} \left( \int_\Omega \psi_{\delta,\eps}^2 w_\delta^2 |\nabla\eta|^{2\alpha} \right)^{\frac{1}{\alpha}} .
		\end{array}
	\end{equation}
We study limits as $\eps \to 0$ and then as $\delta \to 0$. Using \eqref{ass_fdelta} and the second in \eqref{ip_technical}, the following pointwise convergences hold (in the arrow subscript we point out the parameter going to zero): 
	\[
	\begin{array}{l}
	c z_{\delta,\eps} \psi_\delta w_\delta^2 \ \ \uparrow_\eps \ \ c\eta^{2\alpha}(\psi_{\delta})_+^2 w_\delta^2 \ \ \uparrow_\delta \ \ c\eta^{2\alpha}\psi_+^2 w^2 \\[0.3cm]
	\psi_{\delta,\eps}^2 w_\delta^2 \ \ \uparrow_\eps \  \ (\psi_{\delta})_+^2 w_\delta^2 \ \ \uparrow_\delta \ \ \psi_+^2 w^2, \\[0.3cm]
	|f_\delta| z_{\delta,\eps} \psi_\delta w_\delta^2 \to_{\eps} |f_\delta|\eta^{2\alpha}(\psi_\delta)_+^2w_\delta^2 \to_\delta 0, \qquad |f_\delta| z_{\delta,\eps} \psi_\delta w_\delta^2 \le C_{\overline{B_{2R} \cap \Omega}} \cdot \psi_+^2 w^2 
	\end{array}
	\]
Therefore, applying the monotone convergence theorem to the right-hand side and to the addendum with $c$ in the left-hand side, while using \eqref{ass_fdelta}, \eqref{eq_ass} and Lebesgue theorem to the addendum with $f_\delta$, we deduce
	\begin{equation}\label{thegood}
	\begin{array}{lcl}
		\disp c \int_\Omega \eta^{2\alpha} \psi_+^2 w^2 & \le & \disp 
		\disp c \int_\Omega \eta^{2\alpha} \psi_+^2 w^2 + \frac{1}{2} \liminf_{\delta \to 0} \int_{\Omega_{\delta}} \eta^{2\alpha} w_\delta^2 |\nabla\psi_\delta|^2 \\[0.4cm]
		& \leq & \disp 2\alpha^2 \left( \int_\Omega \psi_+^2 w^2 \eta^{2\alpha} \right)^{\frac{\alpha-1}{\alpha}} \left( \int_\Omega \psi_+^2 w^2 |\nabla\eta|^{2\alpha} \right)^{\frac{1}{\alpha}} \, .
	\end{array}
	\end{equation}

\noindent \textbf{Case} $c>0$.\\ 	
Using the properties of $\eta$, 	
	$$
		\int_\Omega \psi_+^2 w^2 \eta^{2\alpha} \geq \int_{B_R} \psi_+^2 w^2 \ge \int_{B_{R_0}} \psi_+^2 w^2 > 0
	$$
and
	\[
		\left( \frac{c}{2\alpha^2} \right)^\alpha \int_{B_R} \psi_+^2 w^2 \leq \int_\Omega \psi_+^2 w^2 |\nabla\eta|^{2\alpha} \leq \frac{1}{R^{2\alpha}} \int_{B_{2R}} \psi_+^2 w^2.
	\]
Defining 
	\[
	I(R) = \int_{B_R} \psi_+^2 w^2,
	\]
we deduce the recursive relation 
	$$
		\left( \frac{2\alpha^2}{c R^2} \right)^\alpha I(2R) \geq I(R) \geq I(R_0) > 0 \, .
	$$
We pass to logarithms with $r=2R$. Then
	\begin{equation} \label{liou_exp}
		\log I(r) + \alpha \log\left( \frac{8\alpha^2}{cr^2} \right) \geq \log I(R_0) \, .
	\end{equation}
	For any $r>2R_0$, this inequality holds for any $\alpha \ge 1$. For every
	$$
		r > R_1 := \max\left\{ 2R_0, \frac{8}{\sqrt{c}} \right\}
	$$
	we can choose
	$$
		\alpha = \frac{r \sqrt{c}}{4} > 2 \, ,
	$$
	so that
	$$
		\alpha\log\left( \frac{8\alpha^2}{cr^2} \right) = -\frac{\sqrt{c}\log 2}{4} r \, .
	$$
	With this choice of $\alpha$, dividing both sides of \eqref{liou_exp} by $r$ we get
	$$
		\frac{1}{r} \log \int_{B_r} \psi_+^2 w^2 \geq \frac{1}{r} \log I(R_0) + \frac{\sqrt{c}\log 2}{4} \qquad \forall \, r > R_1 \, .
	$$
	Since the support of $\psi_+$ is contained in $\Omega$, we obtain \eqref{L2_vol_1}.\\[0.2cm]
\noindent 	\textbf{Case} $c=0$.\\
We consider \eqref{thegood} with $\alpha = 1$, which by our definition of $\eta$ simplifies to 	
	\begin{equation}\label{lasimple}
	\frac{1}{2} \liminf_{\delta \to 0} \int_{B_R} \mathbf{1}_{\Omega_\delta} w_\delta^2 |\nabla\psi_\delta|^2 \le \disp \frac{2}{R^2} \int_{B_{2R}} (\psi_\delta)_+^2 w_\delta^2.
	\end{equation}
Using the second in \eqref{ip_technical}, our assumption \eqref{eq_ass} and $w>0$ on $\Omega$, we deduce from \eqref{lasimple} that $\{(\psi_\delta)_+\}$ is uniformly bounded in $W^{1,2}(\Omega')$ for each open set with compact closure $\Omega' \Subset \Omega \cap B_R$. Standard convergence results imply that $\psi_+ \in W^{1,2}_\loc(\Omega \cap B_R)$ and that, for each $\Omega' \Subset \Omega \cap B_R$ 
	\[
	\int_{\Omega'} w^2 |\nabla\psi_+|^2 \le \liminf_{\delta \to 0} \int_{\Omega'} \mathbf{1}_{\Omega_\delta} w_\delta^2 |\nabla\psi_\delta|^2 .
	\]
Letting $\Omega'$ exhaust $\Omega \cap B_R$, from \eqref{lasimple} we get
	\begin{equation}\label{lasimple_2}
	\frac{1}{2} \int_{B_R} w^2 |\nabla\psi_+|^2 \le \disp \frac{2}{R^2} \int_{B_{2R}} \psi_+^2 w^2.
	\end{equation}
Assume by contradiction that \eqref{L2_vol_0} does not hold, that is, the liminf is zero. Letting $R \to \infty$ along a sequence $\{R_j\}$ such that $\{2R_j\}$ realizes the liminf, we deduce from \eqref{lasimple_2} that $|\nabla \psi_+| \equiv 0$ on $\Omega$. As we are assuming that $\{\psi>0\} \neq \emptyset$, an open-closed argument implies that $\psi$ is constant on $\Omega$ (recall that $\Omega$ is connected).
\end{proof}

\begin{proof}[Proof of Theorem \ref{thm_phS_uniq}]\label{sec_thm3}

	We divide the reasoning into three steps. \\[0.2cm]
\textbf{Step 1.} We let $u\in C^\infty(\overline\Omega)$, $u>0$ be such that $\tilde g = u^{\frac{4}{m-2}} g$. Then, by \cite[Proposition 12]{a21}, under the above conformal change of metric, we have the validity of
	\begin{equation}\label{eq_yamabe}
		c_m \Delta u - S^\varphi u + \tilde S^{\tilde\varphi} u^{\frac{m+2}{m-2}} = 0 \qquad \text{on } \, \Omega,
	\end{equation}
	with $c_m = 4\frac{m-1}{m-2}$. We define
	\begin{equation}
		v = 1-u \ \ \in C^\infty(\overline\Omega)
	\end{equation}
	and using the above, together with \eqref{parOm_cond} i) we compute
	\begin{equation} \label{yam_v}
		c_m \Delta v = -S^\varphi u  + \tilde S^{\tilde\varphi}u^{\frac{m+2}{m-2}} \ge - S^\varphi u \left( 1 - u^{\frac{4}{m-2}} \right) = - c_m\Lambda(x)v
	\end{equation}
on $\Omega$, where we set
	\begin{equation}
		\Lambda(x) = \begin{cases}
			\dfrac{S^\varphi u(x)[1- u^{\frac{4}{m-2}}(x)]}{c_m [1-u(x)]} & \text{if } \, u(x) \neq 1 \\[0.3cm]
			\dfrac{S^\varphi}{m-1} & \text{if } \, u(x) = 1.
		\end{cases}
	\end{equation}
Notice that $\Lambda\in C(\overline{\Omega})$, and that, since $S^\varphi \ge 0$,
	\begin{equation} \label{LamSphi}
		\Lambda(x) \le \frac{S^\varphi}{m-1} \qquad \text{on $\{v>0\} \subset \Omega$, with strict inequality on $\{v>0, S^\varphi > 0\}$}.  
	\end{equation}
To see this, it is enough to observe that $y : (0,+\infty)\setminus\{1\} \to \R$ defined by
\[
	y(t) = \frac{t(t^{\frac{4}{m-2}}-1)}{c_m(t-1)}
\]
satisfies $y(t)\to\frac{1}{m-1}$ as $t\to 1$ and $y(t)<\frac{1}{m-1}$ for $t\in(0,1)$.
On the other hand, the validity of \eqref{parOm_cond} ii) gives $v\equiv 0$ on $\partial\Omega$. Summarizing,
	\begin{equation} \label{v_prob}
		\begin{cases}
			\Delta v + \Lambda(x) v \ge 0 & \text{on } \, \Omega \\
			v \equiv 0 & \text{on } \, \partial\Omega. 
		\end{cases}
	\end{equation}
Our goal is to prove that $v \le 0$ on $\Omega$. Once this is shown, we can conclude as in \cite{bmv}: briefly,  $v \le 0$ implies $u \ge 1$ and thus $\tilde g \ge g$. Also, from $u=1$ on $\partial \Omega$ we get $\partial_\nu u \le 0$, where $\nu$ is the outward pointing unit normal to $\partial \Omega \hookrightarrow (\overline{\Omega},g)$. Recalling the identity
	$$
		u^{\frac{m}{m-2}} H_{\tilde g} = u H_g + \frac{2 \partial_\nu u}{m-2} \qquad \text{on } \, \partial \Omega,
	$$
we deduce that $H_{\tilde g} \le H_g$ on $\partial \Omega$. If either $\partial_\nu u = 0$ somewhere in $\partial \Omega$, or if $u=1$ somewhere in $\Omega$, applying respectively Hopf Lemma or the strong maximum principle to \eqref{eq_yamabe} we conclude $u \equiv 1$ by the connectedness of $\Omega$. Therefore, $\tilde g > g$ and $H_{\tilde g} < H_g$ unless $\tilde g = g$. \\[0.2cm]	
\textbf{Step 2.} To prove that $v \le 0$, set
	$$
		w = e^{-\mu f} - 1 \qquad \text{on } \, \Omega_f \doteq \big\{ x : f(x)<0\big\}
	$$
	and observe that $w>0$ on $\Omega_f$, $w\equiv 0$ on $\partial\Omega_f$. Since $\Omega \subset \Omega_f$, for each $\delta \in (0,1)$ we can define $w_\delta \doteq w+\delta$ and 
	\begin{equation}
		\zeta_\delta = \frac{v}{w_\delta} \qquad \text{on } \, \Omega \, .
	\end{equation}
	By the Einstein-type structure \eqref{Ein-type} we deduce
	\begin{equation}
		\Hess(w) = \mu(1+w)(\Ricc^\varphi - \lambda(x)\metric)
	\end{equation}
	so that, tracing, we obtain
	\begin{equation} \label{Delta_w}
		\Delta w_\delta = \Delta w = \mu(S^\varphi - m\lambda(x))(1+w) \, .
	\end{equation}
	Using \eqref{v_prob} and \eqref{Delta_w} we infer the following chain of inequalities on $\Omega$:
	\[
	\begin{array}{lcl}
	\disp w_\delta^{-2} \diver\big( w_\delta^2 \nabla \zeta_\delta \big) & = & \disp \frac{\Delta v}{w_\delta} - \frac{v}{w_\delta^2} \Delta w_\delta \\[0.4cm]
	& \ge & \disp \left[\frac{\mu(1+w)}{w_\delta}(m\lambda(x) - S^\varphi) - \Lambda(x) \right] \zeta_\delta \\[0.4cm]
	& = & \disp \frac{\mu(1+w)}{w_\delta}\left[m\lambda(x) - S^\varphi - \Lambda(x)\frac{1 - e^{\mu f}}{\mu}\right] \zeta_\delta - \frac{\mu \delta}{w_\delta} \Lambda(x) \zeta_\delta.
	\end{array}
	\]
Assume by contradiction that the set 
	$$
		U \doteq \{ x \in \Omega : v(x)>0\} = \{ x \in \Omega : \zeta_\delta(x) > 0 \} 
	$$
is non-empty. There, inequality \eqref{LamSphi} holds, and in view of \eqref{lambda_eps} we obtain
	\begin{equation}\label{eq_alge}
	\begin{array}{lcl}
	\disp m\lambda(x) - S^\varphi - \Lambda(x)\frac{1 - e^{\mu f}}{\mu} & \ge & \disp m\lambda(x) - S^\varphi - \frac{S^\varphi}{m-1}\frac{1 - e^{\mu f}}{\mu} \\[0.4cm]
	& \ge & \disp m\eps \qquad \text{on } \, U.
	\end{array}
	\end{equation}
Therefore, using $\mu(1+w)/w_\delta \ge \mu$ since $\delta < 1$, we conclude
	\[
	w_\delta^{-2} \diver\big( w_\delta^2 \nabla \zeta_\delta \big) \ge \left(\mu m\eps - \frac{\mu \delta}{w_\delta} \Lambda(x)\right) \zeta_\delta \qquad \text{on } \, U \, .
	\]
\vspace{0.2cm}
\noindent \textbf{Step 3. } We shall apply Lemma \ref{L2_liou} on $\Omega$, with the choices 
	\[
	\Omega_\delta = U, \qquad \psi_\delta = \zeta_\delta, \qquad c = \mu m \eps, \qquad f_\delta = \frac{\mu \delta}{w_\delta} \Lambda.
	\]
Note that $|f_\delta| \le \mu |\Lambda|$, $f_\delta \to 0$ on $\Omega$,  
	\[
	\zeta_\delta = 0 \quad \text{on } \, \partial \Omega, \qquad \zeta_\delta w_\delta = v \ \ \forall \, \delta
	\]
and that $\zeta_\delta \to \zeta \doteq v/w$ locally uniformly. The assumptions in the lemma are therefore satisfied, and we have either the validity of \eqref{L2_vol_1}, or the validity of \eqref{L2_vol_0} unless $\zeta$ is constant on $\Omega$. However, since $w^2\zeta^2_+ = v_+^2 < 1$ this contradicts, respectively, assumption \eqref{phS_vol_1} or \eqref{phS_vol_0}. It remains to rule out the possibility that $\zeta$ is a  (positive) constant in $\Omega$, namely, that $v$ is a positive multiple of $w$. In this case, the boundary condition on $v$ and $w$ implies that $\partial \Omega \subset \partial \Omega_f$, whence $\Omega$ is a  connected component of $\Omega_f$. Repeating \eqref{eq_alge} with $\delta = 0$, from \eqref{LamSphi} and \eqref{lambda_eps} we get on $\Omega$ the following inequalities:
	\[
	\begin{array}{lcl}
	0 = \disp w^{-2} \diver\big( w^2 \nabla \zeta \big) & \ge & \disp \left[\frac{\mu(1+w)}{w}(m\lambda(x) - S^\varphi) - \Lambda(x) \right] \zeta \\[0.4cm]
	& = & \disp \frac{\mu(1+w)}{w}\left[m\lambda(x) - S^\varphi - \Lambda(x)\frac{1 - e^{\mu f}}{\mu}\right] \zeta \\[0.4cm]
	& \ge & \disp \frac{\mu(1+w)}{w}\left[m\lambda(x) - S^\varphi - \frac{S^\varphi}{m-1}\frac{1 - e^{\mu f}}{\mu}\right] \zeta \ge 0.
	\end{array}
	\]
Whence, all are equalities and 
	\begin{equation}\label{eq_rigi}
	\Lambda(x) = \frac{S^\varphi}{m-1} = \frac{\mu}{1-e^{\mu f}}(m\lambda(x) - S^\varphi) \qquad \text{on } \, \Omega.
	\end{equation}
Comparing \eqref{v_prob} and \eqref{Delta_w}, we get that \eqref{v_prob} is satisfied with equality sign, which because of \eqref{yam_v} implies $\tilde S^{\tilde \varphi} = S^\varphi$ on $\Omega$. Also, the first identity in \eqref{eq_rigi}, inequality \eqref{LamSphi} and $u = 1-v < 1$ imply $S^\varphi \equiv 0$ on $\Omega$. 
However, in this case $\Lambda(x) = 0$ on $\Omega$ and $v$ would be a positive, bounded harmonic function which vanishes on the boundary of $\Omega$; hence, its extension $\bar v$ with zero on $M\backslash \Omega$ would be weakly subharmonic, non-negative and, from $v \le 1$, it would satisfy
	\[
	\liminf_{r \to \infty} \frac{1}{r^2} \int_{B_r} {\bar v}^2 \le \liminf_{r \to \infty} \frac{|\Omega \cap B_r|}{r^2} = 0.
	\]	
By Yau's theorem (in the improved version given by \cite[Thm. A]{karp}) we conclude $\bar v \equiv 0$, contradiction.  
\end{proof}

\section{Proof of Theorem \ref{thm_Sph}} \label{sec_thm_Sph}

Hereafter, we let $\{e_i\}$, $1 \le i \le m$ be a local orthonormal frame on $M$, with dual coframe $\{\theta^j\}$. Let also $\{E_a\}$, $1 \le a \le n$ be a local orthonormal frame on $N$. Given a smooth map $\varphi : M \to N$, we write in components the differential $\di \varphi$ and the Hessian $\nabla \di \varphi$ (see \cite{el88}) as
	\[
	\di \varphi = \varphi^a_i \theta^i \otimes E_a, \qquad \nabla \di \varphi = \varphi^a_{ij} \theta^j \otimes \theta^i \otimes E_a.
	\]
Hence, the energy density $|\di \varphi|^2$ and the tension field $\tau(\varphi) = \mathrm{Tr}(\nabla \di \varphi)$ are given by
	\[
	|\di \varphi|^2 = \varphi^a_i\varphi^a_i, \qquad \tau(\varphi) = \varphi^a_{ii}E_a
	\] 
where Einstein convention on repeated indices is tacitly assumed. Also, hereafter, the presence of commas in the subscript of the components of a tensor means taking covariant derivatives.\par

We first prove that a $\varphi$-CPE structure has necessarily constant scalar curvature.

\begin{proposition} \label{prop_DS}
	Let $(M,\metric)$ be a connected manifold of dimension $m\geq 3$ with a $\varphi$-CPE structure
	\begin{equation} \label{phCPE1}
		\begin{cases}
			\Hess(w) - w\left( \Ricc^\varphi - \frac{S^\varphi}{m-1} \metric \right)= T^\varphi \\
			(1+w)\tau(\varphi) = -\di\varphi(\nabla w)
		\end{cases}
	\end{equation}
	for some $\varphi : (M,\metric) \to (N,\metric_N)$ and $\alpha\in\RR$. Then, $S^\varphi$ is constant on $M$. Furthermore, if $M$ is compact and $w$ is non-constant then $S^\varphi > 0$.
\end{proposition}

\begin{proof}
We prove the identity 	
	\begin{equation} \label{DSphi1}
		\left( \frac{m-2}{m} + w \right) \nabla S^\varphi = 0 \qquad \text{on } \, M \, .
	\end{equation}
We take covariant derivative of the first equation in \eqref{phCPE1} to get
	\begin{equation} \label{phCPE2}
		w_{ij,k} - w\left( R^\varphi_{ij,k} - \frac{S^\varphi_k}{m-1}\delta_{ij} \right) - \left( R^\varphi_{ij} - \frac{S^\varphi}{m-1} \delta_{ij} \right) w_k - R^\varphi_{ij,k} + \frac{S^\varphi_k}{m} \delta_{ij} = 0 \, .
	\end{equation}
	We recall the $\varphi$-Schur's identity (see equation (2.10) of \cite{acr21})
	\begin{equation}\label{phi_Schur_gen}
		R^\varphi_{ij,i} = \frac{1}{2} S^\varphi_j - \alpha \varphi^a_{ss}\varphi^a_j
	\end{equation}
	and we trace \eqref{phCPE2} with respect to $i$ and $k$ to infer
	\begin{align*}
		0 & = w_{ij,i} - w R^\varphi_{ij,i} + w \frac{S^\varphi_j}{m-1} - R^\varphi_{ij} w_i + \frac{S^\varphi}{m-1} w_j - R^\varphi_{ij,i} + \frac{S^\varphi_j}{m} \\
		& = w_{ii,j} + w_t R_{tiji} - w \left( \frac{1}{2} S^\varphi_j - \alpha \varphi^a_{kk}\varphi^a_j \right) + \frac{1}{m-1} w S^\varphi_j + \frac{1}{m-1} S^\varphi w_j \\
		& \phantom{=\;} - R^\varphi_{ij} w_i - \left( \frac{1}{2} S^\varphi_j - \alpha \varphi^a_{kk} \varphi^a_j \right) + \frac{S^\varphi_j}{m} \, .
	\end{align*}
	Tracing the first in \eqref{phCPE1} we obtain
	\begin{equation} \label{phCPE3}
		\Delta w + \frac{S^\varphi}{m-1} w = 0 \, .
	\end{equation}
	Thus, using \eqref{phCPE3} and the second in \eqref{phCPE1}, we deduce
	\begin{align*}
		0 & = - \frac{1}{m-1} S^\varphi_j w - \frac{1}{m-1} S^\varphi w_j + w_t R^\varphi_{tj} + \alpha \varphi^a_t \varphi^a_j w_t - \frac{1}{2} S^\varphi_j w + \alpha \varphi^a_{kk} \varphi^a_j w \\
		& \phantom{=\;} + \frac{1}{m-1} S^\varphi_j w + \frac{1}{m-1} S^\varphi w_j - R^\varphi_{tj} w_t - \frac{1}{2} S^\varphi_j + \alpha \varphi^a_{kk} \varphi^a_j + \frac{1}{m} S^\varphi_j \\
		& = - \frac{1}{2} (1+w) S^\varphi_j + (1+w) \alpha \varphi^a_{kk} \varphi^a_j + \frac{1}{m} S^\varphi_j + \alpha \varphi^a_t \varphi^a_j w_t \\
		& = - \frac{1}{2} \left( \frac{m-2}{m} + w \right) S^\varphi_j \, ,
	\end{align*}
	that is, \eqref{DSphi1}.\par
Next, we observe that \eqref{DSphi1} implies $\nabla S^\varphi \equiv 0$ on the open subset $U = \{ w \neq - (m-2)/m\}$. On the other hand, if $x_0 \in \mathrm{Int}(M \backslash U)$, in the sense that there exists $\eps>0$ such that $B_\eps(x_0) \subset M \backslash U$, from $w = -(m-2)/m$ on $B_\eps(x_0)$, \eqref{phCPE3} and $m \ge 3$ we get $S^\varphi \equiv 0$ on $B_\eps(x_0)$. Concluding, $\nabla S^\varphi \equiv 0$ on $U \cup \mathrm{Int}(M \backslash U)$, the complementary of which is a closed set with empty interior. Hence, $\nabla S^\varphi \equiv 0$ on $M$ and thus $S^\varphi$ is constant.\par
To conclude, if $M$ is compact and $w$ is non-constant, then integrating \eqref{phCPE3} against $w$ we get \eqref{eq_intide} (with $S^\varphi$ in place of $S$), which forces $S^\varphi>0$.
\end{proof}

The next result is a Kazdan-Warner type obstruction that holds on every $\varphi$-CPE structure, which for convenience we write as 
	\begin{equation}\label{phiCPE_nelladim_1}
	\begin{cases}
	w_{ji} = (1+w) T^\varphi_{ji} - w \frac{S^\varphi}{m(m-1)} \delta_{ji} \\
		\varphi^a_s w_s = - (1+w) \varphi^a_{tt} \, .
	\end{cases}
	\end{equation}

For constant $\varphi$, the identity reduces to the formula in \cite{hwang00} recalled in the Introduction. 

%

\begin{lemma} \label{lem_divT}
	Let $(M,\metric)$ be a manifold of dimension $m\geq 2$ with a $\varphi$-CPE structure as in \eqref{phCPE}. Then
	\begin{equation} \label{div_Tsharp}
		\div(T^\varphi(\nabla w,\,\cdot\,)^\sharp) = \alpha(1+w)|\tau(\varphi)|^2 + (1+w)|T^\varphi|^2 \, .
	\end{equation}
In particular, if $M$ is compact,
	\begin{equation} \label{div_T_int}
		\int_M (1+w)|T^\varphi|^2 = - \alpha \int_M (1+w)|\tau(\varphi)|^2 \, .
	\end{equation}
\end{lemma}

\begin{proof}
	We compute
	\begin{align*}
		\div(T^\varphi(\nabla w,\,\cdot\,)^\sharp) & = (T^\varphi_{ik} w_k)_i \\
		& = T^\varphi_{ik,i} w_k + T^\varphi_{ik} w_{ki} \\
		& = \left(R^\varphi_{ik,i} - \frac{S^\varphi_i}{m} \delta_{ik}\right) w_k + T^\varphi_{ik} w_{ki} \, .
	\end{align*}
Using the $\varphi$-Schur's identity \eqref{phi_Schur_gen} and the constancy of $S^\varphi$ which follows from Proposition \ref{prop_DS}, 
	$$
		\div(T^\varphi(\nabla w,\,\cdot\,)^\sharp) = - \alpha\varphi^a_{tt}\varphi^a_k w_k + T^\varphi_{ik} w_{ki} \, .
	$$
	Then, the validity of \eqref{phiCPE_nelladim_1} gives \eqref{div_Tsharp}. Equation \eqref{div_T_int} follows immediately from \eqref{div_Tsharp}.
\end{proof}

As in \cite{bal19,bal20}, the proof of Theorem \ref{thm_Sph} depends on an integral identity, \eqref{int_form2} below, obtained by comparing two different Bochner formulas. Before, we need to recall a few other facts and definitions. Although not strictly necessary in what follows, but to simplify notations, we introduce the linear map
$$
	\mathscr W^\varphi : S^2_0(M) \to S^2_0(M) \, 
$$
on the space $S^2_0(M)$ of traceless $2$-covariant, symmetric tensors on $M$, defined, for $\beta = \beta_{ij} \, \theta^i \otimes \theta^j \in S^2_0(M)$, by setting
\begin{equation}
	\mathscr W^\varphi(\beta) = \left[ W^\varphi_{tikj} - \frac{\alpha}{2} \varphi^a_t \left( \varphi^a_i \delta_{kj} + \varphi^a_j \delta_{ki} \right) \right] \beta_{tk} \, \theta^i \otimes \theta^j \, .
\end{equation}
Obviously indices $1\leq a,b,\dots\leq n = \dim N$ and $1\leq i,j,\dots\leq m$ refer to local orthonormal coframes respectively on $N$ and $M$. Note that $\mathscr W^\varphi$ is well defined and self-adjoint with respect to the standard extension of $\metric$ to $S^2_0(M)$, that we will denote with the same symbol. This is crucial for the validity of inequality \eqref{WT_ineq} that we shall use later.

We let $C^\varphi$ be the $\varphi$-Cotton tensor, defined as the obstruction to the $\varphi$-Schouten tensor $A^\varphi$ in \eqref{def_Schou} to be Codazzi. Thus, its components in a local orthonormal coframe are given by
$$
	C^\varphi_{ijk} = A^\varphi_{ij,k} - A^\varphi_{ik,j} \, .
$$
A calculation in \cite{acr21} shows the validity of the following symmetries:
\begin{equation} \label{C_sym}
	\begin{cases}
		C^\varphi_{ijk} = - C^\varphi_{ikj} \qquad \text{and thus } \, C^\varphi_{ikk} = 0 \\
		C^\varphi_{kki} = \alpha \varphi^a_{kk} \varphi^a_i \\
		C^\varphi_{ijk} + C^\varphi_{jki} + C^\varphi_{kij} = 0 \, .
	\end{cases}
\end{equation}

Our argument to prove Theorem \ref{thm_Sph} keeps the same guidelines as \cite{bal19}, but with some simplifications. We split it into some lemmas. The first step is the following Bochner identity:

\begin{lemma} \label{lem_Boch}
	Let $(M,\metric)$ be a manifold of dimension $m\geq 3$, let $\varphi : (M,\metric) \to (N,\metric_N)$ be a smooth map and let $\alpha\in\RR$. Assume $S^\varphi$ is constant. Then
	\begin{equation} \label{Boch_Ric}
		\begin{split}
			\frac{1}{2}\Delta|T^\varphi|^2 & = |\nabla T^\varphi|^2 + \frac{m}{m-2} \tr(T^\varphi)^3 + \frac{1}{m-1} S^\varphi|T^\varphi|^2 \\
			& \phantom{=\;} - \langle \mathscr W^\varphi(T^\varphi), T^\varphi \rangle + \left( C^\varphi_{ijk} R^\varphi_{ij} \right)_k - \frac{1}{2}|C^\varphi|^2 - C^\varphi_{kki,j} R^\varphi_{ij} \, .
		\end{split}
	\end{equation}
\end{lemma}

\begin{remark}
\emph{From now on we indicate a $2$-covariant tensor and its corresponding endomorphism with the same letter. Thus $(T^\varphi)^3$ means the composition of endomorphisms $T^\varphi \circ T^\varphi \circ T^\varphi$.
}
\end{remark}

\begin{proof}
	From equation (3.6) of \cite{acr21} we have
	\begin{equation} \label{Boch_ACR}
		\begin{split}
			\frac{1}{2}\Delta|T^\varphi|^2 & = |\nabla T^\varphi|^2 + \frac{m-2}{2(m-1)} \tr(T^\varphi \circ \Hess(S^\varphi)) + \frac{m}{m-2} \tr(T^\varphi)^3 + \frac{S^\varphi}{m-1} |T^\varphi|^2 \\
			& \phantom{=\;} + \tr(\div C^\varphi \circ T^\varphi) - \langle \mathscr W^\varphi(T^\varphi), T^\varphi \rangle - \tr(T^\varphi \circ \nabla \tr C^\varphi)
		\end{split}
	\end{equation}
	where we have set
	\begin{equation} \label{divC}
		\div C^\varphi = C^\varphi_{ijk,k} \, \theta^i \otimes \theta^j \, , \qquad \tr C^\varphi = C^\varphi_{kki} \, \theta^i \,.
	\end{equation}
	Since $S^\varphi$ is constant, $\Hess(S^\varphi) = 0$, while using \eqref{divC} we deduce
	\begin{align}
		\label{divC1}
		\tr(T^\varphi\circ\nabla\tr C^\varphi) & = C^\varphi_{kki,j} R^\varphi_{ij} - \frac{S^\varphi}{m} C^\varphi_{kki,i} \\
		\label{divC2}
		\tr(\div C^\varphi\circ T^\varphi) & = C^\varphi_{ijk,k} R^\varphi_{ij} - \frac{S^\varphi}{m} C^\varphi_{ssk,k} \, .
	\end{align}
Again using the constancy of $S^\varphi$ and the relation between $R^\varphi_{ij,k}$ and $C^\varphi_{ijk}$ we get
	\begin{equation}
		C^\varphi_{ijk,k} R^\varphi_{ij} = (C^\varphi_{ijk} R^\varphi_{ij})_k - \frac{1}{2}|C^\varphi|^2 \, .
	\end{equation}
Inserting the above informations into \eqref{Boch_ACR} we conclude \eqref{Boch_Ric}.
\end{proof}

Note that the validity of Lemma \ref{lem_Boch} is independent of that of the $\varphi$-CPE structure \eqref{phCPE}. Now the idea is to make formula \eqref{Boch_Ric} interact with $1+w$ in order to be able to use \eqref{div_T_int}. Towards this aim we observe that tracing the first equation in \eqref{phiCPE_nelladim_1} we obtain
\begin{equation} \label{eq1}
	\frac{1}{2} \Delta(1+w)^2 = - \frac{1}{m-1} S^\varphi w (1+w) + |\nabla w|^2 \, ,
\end{equation}
thus, when $M$ is compact, integrating against $|T^\varphi|^2$ gives
$$
	\int_M |T^\varphi|^2 |\nabla w|^2 = \int_M \frac{S^\varphi}{m-1}(1+w)w|T^\varphi|^2 + \frac{1}{2} \int_M (1+w)^2 \Delta|T^\varphi|^2.
$$
We insert \eqref{Boch_Ric} into the above and integrate by parts the term with $\left( C^\varphi_{ijk} R^\varphi_{ij} \right)_k$ to obtain
\begin{equation} \label{eq3}
	\begin{split}
		\int_M |T^\varphi|^2 |\nabla w|^2 & = \int_M (1+w)^2 |\nabla T^\varphi|^2 + \frac{m}{m-2} \int_M (1+w)^2 \tr(T^\varphi)^3 \\
		& \phantom{=\;} + \frac{2}{m-1} \int_M (1+w)^2 S^\varphi|T^\varphi|^2 - \int_M (1+w)^2 \langle \mathscr W^\varphi(T^\varphi), T^\varphi \rangle \\
		& \phantom{=\;} - \frac{S^\varphi}{m-1} \int_M (1+w)|T^\varphi|^2 - \frac{1}{2} \int_M (1+w)^2 |C^\varphi|^2 \\
		& \phantom{=\;} - \int_M (1+w)^2 C^\varphi_{kki,j} R^\varphi_{ij} - 2\int_M (1+w)w_k C^\varphi_{ijk} R^\varphi_{ij} \, .
	\end{split}
\end{equation}

So far, to get \eqref{eq3} we did not use \eqref{phCPE} in its full strength, just \eqref{eq1} and the constancy of $S^\varphi$. Hereafter, we shall exploit all of the assumptions in Theorem \ref{thm_Sph}, that is, the validity of the whole of \eqref{phiCPE_nelladim_1} with $\tau(\varphi) = 0$. We first get
	\begin{equation}\label{eq_tauvarphiu0}
	\varphi^a_s w_s = - (1+w) \varphi^a_{tt} = 0 
	\end{equation}
and secondly, from the constancy of $S^\varphi$ and again by $\tau(\varphi) = 0$, the $\varphi$-Schur identity \eqref{phi_Schur_gen} becomes 
\begin{equation} \label{phi_Schur}
	T^\varphi_{ji,i} = R^\varphi_{ji,i} = \frac{1}{2} (S^\varphi)_j - \alpha \varphi^a_{tt}\varphi^a_j = 0.
\end{equation}	
Taking covariant derivative of the CPE equation \eqref{phiCPE_nelladim_1} we have
\begin{equation} \label{eq10}
	(1+w) R^\varphi_{kj,i} = w_{kj,i} - w_i R^\varphi_{jk} + \frac{S^\varphi}{m-1} w_i \delta_{jk} \, .
\end{equation}
Interchanging the role of $j$ and $i$, subtracting the two identities and
using Ricci commutation relations for $w_{kj,i}$ we infer
\begin{equation} \label{eq12_0}
	(1+w) (R^\varphi_{kj,i} - R^\varphi_{ki,j}) = w_t R_{tkji} + \frac{S^\varphi}{m-1} ( w_i \delta_{jk} - w_j \delta_{ik} ) - (w_i R^\varphi_{jk} - w_j R^\varphi_{ik}) \, .
\end{equation}
Taking into account that $S^\varphi$ is constant, $C^\varphi_{kji} = R^\varphi_{kj,i} - R^\varphi_{ki,j}$ and we obtain
\begin{equation} \label{eq12_01}
	(1+w)C^\varphi_{kji} = w_t R_{ijkt} + \frac{S^\varphi}{m-1} ( w_i \delta_{jk} - w_j \delta_{ik} ) - (w_i R^\varphi_{jk} - w_j R^\varphi_{ik}) \, .
\end{equation}
We multiply the above relation by $R^\varphi_{kj}$ and take divergence to get
	\begin{equation}\label{eq_diverY}
	\begin{array}{l}
	\disp \big( w_j R^\varphi_{ik}R^\varphi_{kj} + w_t R_{ijkt} R^\varphi_{jk}\big)_i \\[0.3cm]
	\qquad = \left((1+w)C^\varphi_{kji}R^\varphi_{kj} + w_i \left( |\Ric^\varphi|^2 - \frac{(S^\varphi)^2}{m-1} \right) + \frac{S^\varphi  w_j R^\varphi_{ji}}{m-1} \right)_i \\[0.4cm]
	\qquad = \left((1+w)C^\varphi_{kji}R^\varphi_{kj} + w_i |T^\varphi|^2 + \frac{S^\varphi  w_j T^\varphi_{ji}}{m-1} \right)_i \\[0.4cm]	
		\qquad = \left((1+w) C^\varphi_{kji}R^\varphi_{kj}\right)_i + \langle \nabla w, \nabla|T^\varphi|^2 \rangle + \frac{S^\varphi}{m-1} |T^\varphi|^2, 	
	\end{array}
	\end{equation}
where in the last equality we used \eqref{phCPE} in the form \eqref{phiCPE_nelladim_1}, its trace
\begin{equation} \label{Delta_w_1}
	\Delta w = - \frac{S^\varphi}{m-1} w
\end{equation}
and the $\varphi$-Schur identity \eqref{phi_Schur}. We examine the left hand side of \eqref{eq_diverY}, that is, 
	\[
	(\ast) \doteq \big( w_j R^\varphi_{ik}R^\varphi_{kj} + w_t R_{ijkt}R^\varphi_{jk}\big)_i.
	\]
Expanding the divergence, 
\begin{align*}
	(\ast) & = w_{ij} R^\varphi_{ik} R^\varphi_{kj} + w_j R^\varphi_{ik} R^\varphi_{kj,i} + w_j R^\varphi_{ik,i} R^\varphi_{kj} \\
	& \phantom{=\;} + w_{ti} R_{ijkt} R^\varphi_{jk} + w_t R_{ijkt,i} R^\varphi_{jk} + w_t R_{ijkt} R^\varphi_{jk,i} \, .
\end{align*}
Tracing the second Bianchi identity, notice that
\begin{equation} \label{eq15}
	R_{ijkt,i} = R^\varphi_{jt,k} - R^\varphi_{jk,t} + \alpha(\varphi^a_{jk} \varphi^a_t - \varphi^a_{jt} \varphi^a_k) \, .
\end{equation}
Hence, using \eqref{eq_tauvarphiu0}, \eqref{phi_Schur} and \eqref{eq15}, 
\begin{align*}
	(\ast) & = w_j R^\varphi_{ik} ( R^\varphi_{kj,i} - R^\varphi_{ki,j} + R^\varphi_{ki,j} ) + w_t R^\varphi_{jk} ( R^\varphi_{jt,k} - R^\varphi_{jk,t}) \\
	& \phantom{=\;} + \alpha w_t R^\varphi_{jk} (\varphi^a_{jk} \varphi^a_t - \varphi^a_{jt} \varphi^a_k) + \frac{1}{2} w_t R_{ijkt} (R^\varphi_{jk,i} - R^\varphi_{ik,j}) \\
	& \phantom{=\;} + w_{ij} R^\varphi_{ik} R^\varphi_{kj} + w_{ti} R_{ijkt} R^\varphi_{jk} \\
	& = \frac{1}{2} \langle \nabla w,\nabla|\Ricc^\varphi|^2 \rangle + w_j C^\varphi_{kji} R^\varphi_{ik} + \frac{1}{2} w_t R_{ijkt} C^\varphi_{kji} \\
	& \phantom{=\;} + w_{ij} R^\varphi_{ik} R^\varphi_{kj} + w_{ti} R_{ijkt} R^\varphi_{jk} -\alpha w_t R^\varphi_{jk}\varphi^a_{jt} \varphi^a_k + w_t R^\varphi_{jk} C^\varphi_{jtk} \\
	& = 2w_jC^\varphi_{kji}R^\varphi_{ik} + \frac{1}{2} \langle \nabla w,\nabla |T^\varphi|^2 \rangle + \frac{1}{2} w_t R_{ijkt} C^\varphi_{kji} \\
	& \phantom{=\;} + w_{ij} R^\varphi_{ik} R^\varphi_{kj} + w_{ti} R_{ijkt} R^\varphi_{jk} - \alpha w_t R^\varphi_{jk}\varphi^a_{jt} \varphi^a_k \, .
\end{align*}
We next exploit \eqref{eq12_01} to remove the term $w_t R_{ijkt}$. Because of the second in \eqref{C_sym}, which in our setting becomes $C^\varphi_{kki}=0$, we obtain 
	\begin{align*}
	(\ast) & = 2w_jC^\varphi_{kji}R^\varphi_{ik} + \frac{1}{2} \langle \nabla w,\nabla |T^\varphi|^2 \rangle \\ 
	& \phantom{=\;} + \frac{1}{2}\left[(1+w)C^\varphi_{kji} - \frac{S^\varphi}{m-1} ( w_i \delta_{jk} - w_j \delta_{ik} ) + (w_i R^\varphi_{jk} - w_j R^\varphi_{ik})\right] C^\varphi_{kji} \\
	& \phantom{=\;} + w_{ij} R^\varphi_{ik} R^\varphi_{kj} + w_{ti} R_{ijkt} R^\varphi_{jk} - \alpha w_t R^\varphi_{jk}\varphi^a_{jt} \varphi^a_k \\
	& = w_jC^\varphi_{kji}R^\varphi_{ik} + \frac{1}{2} \langle \nabla w,\nabla |T^\varphi|^2 \rangle + \frac{1}{2}(1+w)|C^\varphi|^2  \\ 
	& \phantom{=\;} + w_{ij} R^\varphi_{ik} R^\varphi_{kj} + w_{ti} R_{ijkt} R^\varphi_{jk} - \alpha w_t R^\varphi_{jk}\varphi^a_{jt} \varphi^a_k \, .
	\end{align*}
Using the Ricci commutation relations for the tensor $\Ricc^\varphi$:
$$
	R^\varphi_{st,ji} = R^\varphi_{st,ij} + R^\varphi_{lt} R_{lsji} + R^\varphi_{sl} R_{ltji} \, ,
$$
and the $\varphi$-Schur identity \eqref{phi_Schur}, which implies $R^\varphi_{ik,kt} = 0$, we deduce 
	\[
	\begin{array}{lcl}
	R^\varphi_{jk} R_{ijkt} &= & \disp R^\varphi_{jk} R_{jitk} = R^\varphi_{ik,tk} - R^\varphi_{ij} R_{jt} \\[0.2cm]
	& = & \disp R^\varphi_{ik,tk} - R^\varphi_{ij}R^\varphi_{jt} - \alpha R^\varphi_{ij} \varphi^a_j \varphi^a_t \, .
	\end{array}
	\]
Plugging into the above, we get
	\begin{align*}
	(\ast) & = w_jC^\varphi_{kji}R^\varphi_{ik} + \frac{1}{2} \langle \nabla w,\nabla |T^\varphi|^2 \rangle + \frac{1}{2}(1+w)|C^\varphi|^2  \\ 
	& w_{ti}R^\varphi_{ik,tk} - \alpha w_{ti} R^\varphi_{ij} \varphi^a_j \varphi^a_t - \alpha w_t R^\varphi_{jk}\varphi^a_{jt} \varphi^a_k \, .
	\end{align*}
Differentiating \eqref{eq_tauvarphiu0} we get 
	\[
	w_t \varphi^a_{ti} = - w_{ti}\varphi^a_t,
	\]
and therefore 
	\[
	w_{ti} R^\varphi_{ij} \varphi^a_j \varphi^a_t + w_t R^\varphi_{jk}\varphi^a_{jt} \varphi^a_k = \disp - w_t \varphi^a_{ti}R^\varphi_{ij} \varphi^a_j + w_t R^\varphi_{jk}\varphi^a_{jt} \varphi^a_k = 0. 
	\]
Inserting into the above, we infer
	\[
	(\ast) = w_jC^\varphi_{kji}R^\varphi_{ik} + \frac{1}{2} \langle \nabla w,\nabla |T^\varphi|^2 \rangle + \frac{1}{2}(1+w)|C^\varphi|^2 + w_{ti}R^\varphi_{ik,tk}.
	\]
Plugging into \eqref{eq_diverY} and rearranging, we obtain
	\begin{equation}\label{eq_diverY_2}
	\begin{array}{l}
		\disp w_jC^\varphi_{kji}R^\varphi_{ik} + \frac{1}{2}(1+w)|C^\varphi|^2 + w_{ti}R^\varphi_{ik,tk} \\[0.3cm]
		\qquad = \left((1+w) C^\varphi_{kji}R^\varphi_{kj}\right)_i + \dfrac{1}{2} \langle \nabla w, \nabla|T^\varphi|^2 \rangle + \dfrac{S^\varphi}{m-1} |T^\varphi|^2 \, .  	
	\end{array}
	\end{equation}
We study the term $w_{ti}R^\varphi_{ik,tk}$. Again from the $\varphi$-Schur identity \eqref{phi_Schur} and from \eqref{phiCPE_nelladim_1},
	\[
	\begin{array}{lcl}
	w_{ti}R^\varphi_{ik,tk} & = & \disp (1+w) T^\varphi_{it} R^\varphi_{ik,tk} = (1+w) T^\varphi_{it} T^\varphi_{ik,tk} \\[0.2cm]
	& = & \disp \big( (1+w)T^\varphi_{it} T^\varphi_{ik,t} \big)_k - w_k T^\varphi_{it} T^\varphi_{ik,t} - (1+w) T^\varphi_{it,k} T^\varphi_{ik,t} \, .
	\end{array}
	\]
Next, since $S^\varphi$ is constant, $C^\varphi_{ikt} =  T^\varphi_{ik,t} -  T^\varphi_{it,k}$ and we deduce
	\[
	\begin{array}{lcl}
	w_{ti}R^\varphi_{ik,tk} & = & \disp \big( (1+w)T^\varphi_{it}C^\varphi_{ikt} \big)_k + \frac{1}{2}\diver \big( (1+w)\nabla |T^\varphi|^2\big) \\[0.2cm]
	& & \disp - w_k T^\varphi_{it} T^\varphi_{ik,t} - (1+w) T^\varphi_{it,k} T^\varphi_{ik,t} \, .
	\end{array}
	\]
Using the identities
	\[
	\begin{array}{lcl}
	T^\varphi_{it,k}T^\varphi_{ik,t} & = & \disp R^\varphi_{it,k}R^\varphi_{ik,t} = |\nabla\Ricc^\varphi|^2 - \frac{1}{2}|C^\varphi|^2 = |\nabla T^\varphi|^2 - \frac{1}{2}|C^\varphi|^2 \\[0.2cm]
	T^\varphi_{it}T^\varphi_{ik,t} & = & \disp T^\varphi_{it}(T^\varphi_{it,k} - C^\varphi_{itk}) = \frac{1}{2} \big( |T^\varphi|^2\big)_k - R^\varphi_{it}C^\varphi_{itk}
	\end{array} 
	\]
(recall that $C^\varphi_{iik}=0$ by the second in \eqref{C_sym} and $\tau(\varphi)=0$), we conclude
	\[
	\begin{array}{lcl}
	w_{ti}R^\varphi_{ik,tk} & = & \disp \big( (1+w)T^\varphi_{it}C^\varphi_{ikt} \big)_k + \frac{1}{2} \diver \big( (1+w)\nabla |T^\varphi|^2\big) \\[0.2cm]
	& & \disp - \frac{1}{2} \langle \nabla w, \nabla |T^\varphi|^2 \rangle + w_k R^\varphi_{it}C^\varphi_{itk} - (1+w)|\nabla T^\varphi|^2 + \frac{1}{2}(1+w)|C^\varphi|^2 \, .
	\end{array}
	\]
Identity \eqref{eq_diverY_2} therefore becomes
	\[
	\begin{array}{l}
	\disp w_jC^\varphi_{kji}R^\varphi_{ik} + (1+w)|C^\varphi|^2 + \disp \big( (1+w)T^\varphi_{it}C^\varphi_{ikt} \big)_k \\[0.2cm]
 \qquad  \disp + \frac{1}{2} \diver \big( (1+w)\nabla |T^\varphi|^2\big) + w_k R^\varphi_{it}C^\varphi_{itk} - (1+w)|\nabla T^\varphi|^2 \\[0.3cm]
		\qquad = \left((1+w) C^\varphi_{kji}R^\varphi_{kj}\right)_i + \langle \nabla w, \nabla|T^\varphi|^2 \rangle + \dfrac{S^\varphi}{m-1} |T^\varphi|^2 \, .
	\end{array}
	\]
Simplifying and using that, by \eqref{C_sym} and $\tau(\varphi) = 0$, $T^\varphi_{it}C^\varphi_{ikt} = R^\varphi_{it}C^\varphi_{ikt} = - R^\varphi_{it}C^\varphi_{itk}$, 
	\[
	\begin{array}{l}
	\disp (1+w)|C^\varphi|^2 + \frac{1}{2}\diver \big( (1+w)\nabla |T^\varphi|^2\big) - (1+w)|\nabla T^\varphi|^2 \\[0.3cm]
		\qquad = 2\left((1+w) C^\varphi_{kji}R^\varphi_{kj}\right)_i + \langle \nabla w, \nabla|T^\varphi|^2 \rangle + \dfrac{S^\varphi}{m-1} |T^\varphi|^2.	
	\end{array}
	\]
Because of this last identity, we can compute 
	\[
	\begin{array}{lcl}
	\disp \frac{1}{2}\diver \big( (1+w)^2\nabla |T^\varphi|^2\big) & = & \dfrac{1+w}{2} \langle \nabla w, \nabla |T^\varphi|^2 \rangle + \dfrac{1+w}{2} \diver \big( (1+w)\nabla |T^\varphi|^2\big) \\[0.3cm]
	& = & \disp - (1+w)^2|C^\varphi|^2 + (1+w)^2|\nabla T^\varphi|^2 \\[0.3cm]
		& & \disp + 2(1+w)\left((1+w) C^\varphi_{kji}R^\varphi_{kj}\right)_i + \frac{3}{2} (1+w) \langle \nabla w, \nabla|T^\varphi|^2 \rangle \\[0.3cm]
		& & \disp + \frac{S^\varphi}{m-1}(1+w)|T^\varphi|^2. 	
	\end{array}
	\]
Integrating on $M$ and using the divergence theorem, 
	\[
	\begin{array}{lcl}
	0 & = & \disp - \int_M (1+w)^2|C^\varphi|^2 + \int_M (1+w)^2|\nabla T^\varphi|^2 + \frac{S^\varphi}{m-1} \int_M (1+w)|T^\varphi|^2 \\[0.3cm]
		& & \disp - 2 \int_M w_i(1+w) C^\varphi_{kji}R^\varphi_{kj} + \frac{3}{2} \int_M \langle (1+w) \nabla w, \nabla|T^\varphi|^2 \rangle. 	
	\end{array}
	\]
Integrating by parts with the aid of \eqref{eq1}, 
	\[
	\begin{array}{lcl}
	\disp \frac{3}{2} \int_M \langle (1+w) \nabla w, \nabla|T^\varphi|^2 \rangle & = & \disp - \frac{3}{4} \int_M |T^\varphi|^2 \Delta (1+w)^2 \\[0.3cm]
	& = & \disp \frac{3 S^\varphi}{2(m-1)} \int_M w(1+w)|T^\varphi|^2 - \frac{3}{2} \int_M |T^\varphi|^2|\nabla w|^2   	
	\end{array}
	\]
so we finally obtain
	\[
	\begin{array}{lcl}
	0 & = & \disp - \int_M (1+w)^2|C^\varphi|^2 + \int_M (1+w)^2|\nabla T^\varphi|^2 + \frac{S^\varphi}{m-1} \int_M (1+w)|T^\varphi|^2 \\[0.3cm]
		& & \disp - 2 \int_M w_i(1+w) C^\varphi_{kji}R^\varphi_{kj} + \frac{3 S^\varphi}{2(m-1)} \int_M w(1+w)|T^\varphi|^2 - \frac{3}{2} \int_M |T^\varphi|^2|\nabla w|^2 \, .
	\end{array}
	\]
We insert into \eqref{eq3} to remove the term with $C^\varphi_{kji}R^\varphi_{kj}$, and use $C^\varphi_{kki} = 0$, to get 
	\[
	\begin{split}
		\int_M |T^\varphi|^2 |\nabla w|^2 & = \int_M (1+w)^2 |\nabla T^\varphi|^2 + \frac{m}{m-2} \int_M (1+w)^2 \tr(T^\varphi)^3 \\
		& \phantom{=\;} + \frac{2}{m-1} \int_M (1+w)^2 S^\varphi|T^\varphi|^2 - \int_M (1+w)^2 \langle \mathscr W^\varphi(T^\varphi), T^\varphi \rangle \\
		& \phantom{=\;} - \frac{S^\varphi}{m-1} \int_M (1+w)|T^\varphi|^2 - \frac{1}{2} \int_M (1+w)^2 |C^\varphi|^2 + \int_M (1+w)^2|C^\varphi|^2 \\
		& \phantom{=\;} - \int_M (1+w)^2|\nabla T^\varphi|^2 - \frac{S^\varphi}{m-1} \int_M (1+w)|T^\varphi|^2 \\
		& \phantom{=\;} -\frac{3 S^\varphi}{2(m-1)} \int_M w(1+w)|T^\varphi|^2 + \frac{3}{2} \int_M |T^\varphi|^2|\nabla w|^2. 
	\end{split}		
	\]
By the Kazdan-Warner type identity in \eqref{div_T_int} together with assumption $\tau(\varphi) = 0$,
\[
	\int_M (1+w)|T^\varphi|^2 = 0, \quad \text{therefore} \quad \int_M w(1+w)|T^\varphi|^2 = \int_M (1+w)^2|T^\varphi|^2.
\]
Hence, by also using that $S^\varphi$ is constant, after some simplification we eventually get the following integral identity: 
	\begin{equation} \label{int_form2}		
	\begin{split}
		0 &	= \frac{1}{2} \int_M |T^\varphi|^2 |\nabla w|^2  + \frac{m}{m-2} \int_M (1+w)^2 \tr(T^\varphi)^3 +  \frac{1}{2} \int_M (1+w)^2 |C^\varphi|^2 \\
		& \phantom{=\;} + \frac{1}{2(m-1)} \int_M (1+w)^2 S^\varphi|T^\varphi|^2 - \int_M (1+w)^2 \langle \mathscr W^\varphi(T^\varphi), T^\varphi \rangle .
	\end{split}
\end{equation}

We are now ready for the

\begin{proof}[Proof of Theorem \ref{thm_Sph}]
As we have already observed, in our assumptions $S^\varphi$ is a positive constant and \eqref{int_form2} holds. From the proof of Proposition 3.22 in \cite{acr21} we have the validity of the following inequality:
	\begin{equation} \label{WT_ineq}
		\langle \mathscr W^\varphi(T^\varphi), T^\varphi \rangle \leq \sqrt{\frac{m-2}{2(m-1)}} |W^\varphi| |T^\varphi|^2 + \frac{\alpha}{m-2}|\di\varphi|^2 |T^\varphi|^2
	\end{equation}
	and from Okumura's lemma \cite{okumura}, 
	\begin{equation} \label{okumura}
		\tr(T^\varphi)^3 \geq - \frac{m-2}{\sqrt{m(m-1)}} |T^\varphi|^3 \, .
	\end{equation}
	Using \eqref{WT_ineq} and \eqref{okumura}, from \eqref{int_form2} we infer
	\begin{align*}
		0 & \geq \frac{1}{2} \int_M |T^\varphi|^2 |\nabla w|^2 + \frac{1}{2} \int_M (1+w)^2 |C^\varphi|^2 \\
		& \phantom{\geq\;} + \int_M (1+w)^2 \left( \frac{1}{2(m-1)} S^\varphi - \sqrt{\frac{m-2}{2(m-1)}} |W^\varphi| - \frac{\alpha}{m-2}|\di\varphi|^2 - \sqrt{\frac{m}{m-1}} |T^\varphi| \right) |T^\varphi|^2 \, .
	\end{align*}
	Observe that all terms appearing on the RHS of the above inequality are nonnegative, due to \eqref{Sph_pinch}. Hence, they must vanish, and in particular we have
	\begin{equation} \label{int_TDw}
		\int_M |T^\varphi|^2 |\nabla w|^2 = 0 \, .
	\end{equation}
	We claim that this, together with the $\varphi$-CPE equation, implies $T^\varphi\equiv 0$ on $M$. We postpone the proof of this claim to the subsequent Lemma \ref{lem_claim}. Assuming $T^\varphi\equiv0$ on $M$, the $\varphi$-CPE equation gives
	\begin{equation}
		\Hess(w) = - \frac{S^\varphi}{m(m-1)} \metric
	\end{equation}
	with $S^\varphi>0$. But then, Theorem A of Obata, \cite{obata62}, implies that $(M,\metric)$ is isometric to $\Sph^m(\kappa)$ with $\kappa$ as in \eqref{kappa}. In particular, $S = S^\varphi$ and, when $\alpha \neq 0$, this implies $|\di \varphi|^2 = 0$, hence $\varphi$ is constant.
%
%
\end{proof}

\begin{lemma} \label{lem_claim}
	Let $(M,\metric)$ be a compact manifold of dimension $m\geq 3$, let $\varphi : (M,\metric) \to (N,\metric_N)$ be a smooth map and let $\alpha\in\RR$. Assume that $w$ is a non-constant solution to \eqref{phCPE} and 
	\begin{equation} \label{claim_int}
		\int_M |T^\varphi|^2 |\nabla w|^2 = 0 \, .
	\end{equation}
	Then $T^\varphi \equiv 0$ on $M$.
\end{lemma}

\begin{proof}
Observe first that $S^\varphi$ is a positive constant by Proposition \ref{prop_DS}. Following arguments of \cite{hwang00}, we prove that the set
	$$
		C^w_{-1} = \{ x \in M : w(x) = -1 \}
	$$
has zero measure. Since $M$ is compact, $C^w_{-1}$ is compact. Let $\hat C_{-1}$ be the set of critical points of $w$ in $C^w_{-1}$, that is,
	$$
		\hat C_{-1} = \{ x \in C^w_{-1} : \di w_x = 0 \} \, .
	$$
	Note that $C^w_{-1} \setminus \hat C_{-1}$ is a hypersurface with possibly many connected components. Writing the \ref{phCPE} equation in the form \eqref{phiCPE_nelladim_1}, we see that if $p\in\hat C_{-1}$ then for $v\in T_p M$, $v\neq 0$ it holds
	$$
		\Hess(w)(v,v) = \frac{S^\varphi}{m(m-1)} \langle v,v\rangle > 0
	$$
since $S^\varphi>0$. Thus $p$ is a non-degenerate critical point of $s$. Hence the points of $\hat C_{-1}$ are isolated. Since $C^w_{-1}$ is compact, $\hat C_{-1}$ is a finite set. In particular, since $m\geq 3$, $C^w_{-1} \setminus \hat C_{-1}$ is a connected hypersurface and
	$$
		C^w_{-1} = \hat C_{-1} \cup ( C^w_{-1} \setminus \hat C_{-1} )
	$$
	has measure $0$.
	
	We now turn to the proof that $T^\varphi\equiv0$. Since $|T^\varphi||\nabla w| \geq 0$ is a continuous function, \eqref{claim_int} implies that $|T^\varphi||\nabla w| \equiv 0$ on $M$. Hence, the set
	$$
		E = \{ x \in M : T^\varphi_x \neq 0 \}
	$$
	is contained in the set $\{ x \in M : \di w_x = 0 \}$ of critical points of $w$. Note that $E$ is open. Suppose, by contradiction, that $E\neq\emptyset$. Let $x\in E$ and let $U\subseteq E$ be a connected neighbourhood of $x$. Since $\nabla w\equiv 0$ on $E$, there exists a constant $c\in\RR$ such that $w\equiv c$ on $U$. Since $U$ has positive measure, by the previous observation we have $c\neq -1$. Rewriting the $\varphi$-CPE equation in the form
	$$
		(1+w) \Ricc^\varphi - \Hess(w) = \left( \frac{S^\varphi}{m} + \frac{wS^\varphi}{m-1} \right) \metric
	$$
	and using the fact that $w\equiv c\neq -1$ on $U$, we see that
	$$
		\Ricc^\varphi \equiv \frac{S^\varphi}{1+c} \frac{(1+c)m-1}{m(m-1)} \metric \qquad \text{on } \, U \, .
	$$
	In particular, $\Ricc^\varphi$ is a multiple of $\metric$ on $U$ and then $T^\varphi \equiv 0$ on $U$, contradiction.
\end{proof}

\vspace{0.3cm}

\noindent \textbf{Acknowledgements.} The authors are grateful to the anonymous referee for his/her careful work and useful comments, which led to significant improvements in the paper.

\bibliographystyle{plain}

\end{document}